\documentclass[10pt]{amsart}
\usepackage{latexsym, amsmath, amssymb}
\usepackage{version}
\usepackage{epsfig,graphics}
\usepackage{color}
\usepackage{amsthm, amsopn, amsfonts}

\newcommand{\newsection}[1]
{\section{#1}\setcounter{theorem}{0} \setcounter{equation}{0}
\par\noindent}

\newtheorem{theorem}{Theorem}

\newtheorem{lemma}[theorem]{Lemma}
\newtheorem{corollary}[theorem]{Corollary}
\newtheorem{proposition}[theorem]{Proposition}
\newtheorem{definition}[theorem]{Definition}

\newcommand{\la}{\langle}
\newcommand{\ra}{\rangle}

\newcommand{\M}{\mathcal M}

\newcommand{\Qsr}{{\mathcal Q_{sr}}}
\newcommand{\R}{{\mathbb R}}
\newcommand{\tv}{{\tilde{t}}}
\newcommand{\rs}{{r^*}}
\newcommand{\tphi}{{\tilde{\phi}}}

\renewcommand{\S}{{\mathbb S}}

\begin{document}

\title
{
Price's law  on nonstationary spacetimes
}

\author{Jason Metcalfe}
\address{Department of Mathematics, University of North Carolina,
  Chapel Hill, NC  27599-3250}
\email{metcalfe@email.unc.edu}
\author{Daniel Tataru}
\address{Department of Mathematics, University of California,
  Berkeley, CA  94720}
\email{tataru@math.berkeley.edu}
\author{Mihai Tohaneanu}
\address{Department of Mathematics, Purdue University, West Lafayette,
  IN  47907-1395}
\email{mtohanea@math.purdue.edu}

\thanks{The first author was partially supported by NSF grant DMS0800678.
  The second author was partially supported by NSF grant DMS0354539
  and by the Miller Foundation.}

\begin{abstract}
  In this article we study the pointwise decay properties of solutions
  to the wave equation on a class of nonstationary asymptotically flat
  backgrounds in three space dimensions.  Under the assumption that
  uniform energy bounds and a weak form of local energy decay hold
  forward in time we establish a $t^{-3}$ local uniform decay rate
  (Price's law~\cite{MR0376103}) for linear waves.  As a corollary, we
  also prove Price's law for certain small perturbations of the Kerr
  metric.

  This result was previously established by the second author in
  \cite{Tat} on stationary backgrounds.  The present work was motivated
by the problem of nonlinear stability of the Kerr/Schwarzschild
solutions for the vacuum Einstein equations, which seems to require 
a more robust approach to proving linear decay estimates.
 \end{abstract}

\maketitle

\newsection{Introduction}

In this article we consider the question of pointwise decay for
solutions to the wave equation on certain asymptotically flat
backgrounds. Our interest in this problem originates in general
relativity, more precisely the wave equation on Schwarzs\-child and Kerr
backgrounds. There the expected local decay, heuristically derived by
Price~\cite{MR0376103} in the Schwarzschild case, is $t^{-3}$.  This
conjecture was considered independently in two recent articles
\cite{DSS2} and \cite{Tat}.

The work in \cite{DSS2} is devoted to the Schwarzschild space-time,
where separation of variables can be used; in that context, very
precise sharp local decay bounds are established for each of the
spherical modes. Precisely, it is shown that the $k$-th mode leads to
$t^{-3-2k}$ local decay.

The work in \cite{Tat}, on the other hand, applies to a large class of
stationary asymptotically flat space-times, and asserts that if local
energy decay holds then Price's law holds. Further, sharp pointwise
decay rates are established in the full forward light cone; these have
the form $t^{-1}\la t-r\ra^{-2}$. Local energy decay (described later
in the paper) is known to hold in Schwarzschild and Kerr space-times.

Both of the above results involve taking the Fourier transform in
time and hence rely heavily on the stationarity assumption. The aim of this article
is to prove the same result as in \cite{Tat}, namely that local energy decay 
implies Price's law, but without the stationarity  assumption. 
The proof below is more robust than the one in \cite{Tat},
and improves on the classical  vector field method. As an application,
in the last section of the paper we prove that local energy decay 
(and thus Price's law) holds for a class of nonstationary perturbations 
of the Kerr space-time.

Just as in \cite{Tat}, this work is based on the idea that the local
energy estimates contain all the important local information
concerning the flow, and that only leaves the analysis near spatial
infinity to be understood.  In the context of asymptotically flat
metrics, this idea originates in earlier work \cite{gS} and \cite{MT}
where it is proved that local energy decay implies Strichartz
estimates in the asymptotically flat setting, first for the
Schr\"odinger equation and then for the wave equation.  The same
principle was exploited in \cite{MMTT} and \cite{Mihai} to prove
Strichartz estimates for the wave equation on the Schwarzschild and
then on the Kerr space-time.

\subsection{Notations and the regularity of the metric}
\label{ss:cases}
We use $(t=x_0, x)$ for the coordinates in $\R^{1+3}$. We use Latin
indices $i,j=1,2,3$ for spatial summation and Greek indices
$\alpha,\beta=0,1,2,3$ for space-time summation. In $\R^3$ we also use
polar coordinates $x = r \omega$ with $\omega \in \S^2$.  By $\la r
\ra$ we denote a smooth radial function which agrees with $r$ for
large $r$ and satisfies $\la r \ra \geq 2$.  We consider a partition
of $ \R^{3}$ into the dyadic sets $A_R= \{\la r \ra \approx R\}$ for
$R \geq 1$, with the obvious change for $R=1$.
  
To describe the regularity of the coefficients of the metric, we use
the following sets of vector fields:
\[
T = \{ \partial_t, \partial_i\}, \qquad \Omega = \{x_i \partial_j - x_j \partial_i\}, \qquad
S = t \partial_t + x \partial_x,
\]
namely the generators of translations, rotations and scaling. We set
$Z = \{ T,\Omega,S\}$.  Then we define the classes $S^Z(r^k)$ of
functions in $\R^+ \times \R^3$ by
\[
 a \in S^Z(r^k) \Longleftrightarrow 
|Z^j a(t, x)| \leq c_{j} \la r\ra^{k}, \quad j \geq 0.
\]
By $S^Z_{rad}(r^k)$ we denote spherically symmetric functions in $S^Z(r^k)$.

The estimates in this article apply  to solutions for  an inhomogeneous problem
of the form
\begin{equation}\label{box}
 (\Box_g + V)u = f, \qquad u(0) = u_0, \qquad \partial_t u(0)=u_1
\end{equation}
associated to d'Alembertian $\Box_g$ corresponding to a Lorentzian
metric $g$, a potential $V$, nonhomogeneous term $f$ and compactly
supported initial data $u_0$, $u_1$. For the metric $g$ we consider
two cases:

\bigskip

{\bf Case A:} $g$ is a smooth Lorentzian metric  in $\R^+ \times \R^3$, 
with the following properties:

(i) The level sets $t = const$ are space-like.

(ii) $g$ is asymptotically flat in the following sense:
\[
 g = m + g_{sr} + g_{lr}, 
\]
where $m$ stands for the Minkowski metric, $g_{lr}$ is a stationary long range 
spherically symmetric component, with $S^Z_{rad}(r^{-1})$ coefficients,
of the form
\[
 g_{lr} = g_{lr, tt}(r) dt^2 + g_{lr, tr}(r) dt dr + g_{lr, rr}(r) dr^2 +  
g_{lr, \omega \omega}(r) r^2 d \omega^2
\]
and $g_{sr}$ is a short range component of the form
\[
 g_{sr} = g_{sr, tt} dt^2 + 2g_{sr, ti} dt dx_i + g_{sr, ij} dx_i dx_j 
\]
with $S^Z(r^{-2})$ coefficients.

\bigskip

We remark that these assumptions guarantee that $\partial_t$
is time-like near spatial infinity, but not necessarily in a compact set.
This leads us to the second case we consider:

\bigskip

{\bf Case B:} $g$ is a smooth Lorentzian metric in  an exterior domain 
$\R \times \R^3\setminus B(0,R_0)$ which  satisfies
(i),(ii) above in its domain, and in addition

(iii) the lateral boundary $\R \times \partial B(0,R_0)$ is outgoing
space-like.

\bigskip

This latter condition insures that the corresponding wave equation is
well-posed forward in time. This assumption is satisfied in the case
of the Schwarzschild and Kerr metrics (or small perturbations thereof)
in suitable advanced time coordinates. There the parameter $R_0$ is
chosen so that $0 < R_0 < 2M$ in the case of the Schwarzschild metric,
respectively $r^- < R_0 < r^+$ in the case of Kerr (see \cite{HE} for
the definition of $r^{\pm}$), so that the exterior of the $R_0$ ball
contains a neighbourhood of the event horizon.

\subsection{Normalized coordinates}
Our decay results are expressed relative to the distance to the
Minkowski null cone $ \{t = |x|\}$. This can only be done provided
that there is a null cone associated to the metric $g$ which is within
$O(1)$ of the Minkowski null cone.  However, in general the long range
component of the metric produces a logarithmic correction to the
cone. This issue can be remedied via a change of coordinates which
roughly corresponds to using Regge-Wheeler coordinates in
Schwarzschild/Kerr near spatial infinity.  See \cite{Tat}. After a further conformal
transformation\footnote {which changes the potential $V$}, see also
\cite{Tat}, the metric $g$ is reduced to a normal form where
\[
 g_{lr} = g_{\omega}(r) r^2 d \omega^2, \qquad g_\omega \in S^Z_{rad}(r^{-1}).
\]
In particular, we can replace $\Box_g$ by an operator
of the form
\begin{equation}\label{P}
 P = \Box + Q
\end{equation}
where $\Box$ denotes the d'Alembertian in the Minkowski metric and the
perturbation $Q$ has the form
\begin{equation}
Q  = g^{\omega} \Delta_{\omega} +
\partial_\alpha g_{sr}^{\alpha \beta}\partial_\beta + V,
\quad g_{sr}^{\alpha \beta} \in S^Z(r^{-2}), \  g^\omega\in
S^Z_{rad}(r^{-3}), \ V \in S^Z(r^{-3}).
\label{can-p-sr}\end{equation} 
We call these coordinates normal coordinates.  All of the analysis in
the paper is done in normal coordinates and with $g$ in normal form.
The full perturbation $Q$ above has only short range
effects.

\subsection{Uniform energy bounds}
The Cauchy data at time $t$ for the evolution \eqref{box}
is given by $ (u(t),\partial_t u(t))$.  To measure it we use 
the Sobolev spaces $H^k$, with the qualification that in Case A  
this  means $H^k:=H^k(\R^3)$,  while in Case B we 
use the obvious modification $H^k:=H^k(\R^3\backslash B(0,R_0))$.
We begin with the following
definition:

\begin{definition}
We say that the evolution \eqref{box} is forward bounded if the following
estimates hold:
\begin{equation}
\| \nabla u(t_1)\|_{H^k} \leq c_k (\|\nabla u(t_0)\|_{ H^k} +\|f\|_{L^1([t_0,t_1]; H^k)}),
\qquad 0 \leq t_0 \leq t_1, \quad k \geq 0. 
\end{equation}
\end{definition}

It suffices to have this property when $f=0$. Then the $f$ term can be
added in by the Duhamel formula.  One case when the uniform forward
bounds above are easy to establish is when $\partial_t$ is a Killing
vector which is everywhere time-like and $V$ is nonnegative and
stationary. Otherwise, there is no general result, but various cases
have been studied on a case by case basis.

We remark that in the case of the Schwarzschild and Kerr space-times
$\partial_t$ is not everywhere time-like, so the forward boundedness
is not straightforward. However, it is known to hold for Schwarzschild
(see 
\cite{MR2527808} and \cite{MMTT}) as well as for Kerr with small
angular momentum (see \cite{MR2405857}, \cite{DR3}, \cite{TT}) and for
a class of small stationary perturbations of Schwarzschild (see
\cite{DR3}).

The forward boundedness is not explicitly used in what follows, but
it is defined here since it is usually seen as a prerequisite for
everything that follows.

\subsection{Local energy decay}
A stronger property of the wave flow is local energy decay. 
We introduce the local energy norm $LE$ 
\begin{equation}
\begin{split}
 \| u\|_{LE} &= \sup_R  \| \la r\ra^{-\frac12} u\|_{L^2 (\R \times A_R)}  \\
 \| u\|_{LE[t_0, t_1]} &= \sup_R  \| \la r\ra^{-\frac12} u\|_{L^2 ([t_0, t_1] \times A_R)},
\end{split} 
\label{ledef}\end{equation}
its $H^1$ counterpart
\begin{equation}
\begin{split}
  \| u\|_{LE^1} &= \| \nabla u\|_{LE} + \| \la r\ra^{-1} u\|_{LE} \\
 \| u\|_{LE^1[t_0, t_1]} &= \| \nabla u\|_{LE[t_0, t_1]} + \| \la r\ra^{-1} u\|_{LE[t_0, t_1]},
\end{split}
\end{equation}
as well as the dual norm
\begin{equation}
\begin{split}
 \| f\|_{LE^*} &= \sum_R  \| \la r\ra^{\frac12} f\|_{L^2 (\R \times A_R)} \\
 \| f\|_{LE^*[t_0, t_1]} &= \sum_R  \| \la r\ra^{\frac12} f\|_{L^2 ([t_0, t_1] \times A_R)}.
\end{split} 
\label{lesdef}\end{equation}
These definitions are specific to $(1+3)$-dimensional problems.
Some appropriate modifications are needed in other dimensions, see for
instance \cite{MT}.  We also define similar norms for higher Sobolev
regularity
\[
\begin{split}
  \| u\|_{LE^{1,k}} &= \sum_{|\alpha| \leq k} \| \partial^\alpha u\|_{LE^1} \\
  \| u\|_{LE^{1,k}[t_0, t_1]} &= \sum_{|\alpha| \leq k} \| \partial^\alpha u\|_{LE^1[t_0, t_1]} \\
  \| u\|_{LE^{0,k}[t_0, t_1]} &= \sum_{|\alpha| \leq k} \| \partial^\alpha u\|_{LE[t_0, t_1]},
\end{split}
\]
respectively 
\[
\begin{split}
  \| f\|_{LE^{*,k}} &=  \sum_{|\alpha| \leq k}  \| \partial^\alpha f\|_{LE^{*}} \\
  \| f\|_{LE^{*,k}[t_0, t_1]} &=  \sum_{|\alpha| \leq k}  \| \partial^\alpha f\|_{LE^{*}[t_0, t_1]}.
\end{split}  
\]

In Case A above this leads to the following

\begin{definition}
We say that the evolution \eqref{box} has the local energy decay
property if the following estimate holds:
\begin{equation}
 \| u\|_{LE^{1,k}[t_0,\infty)} 
\leq c_k (\|\nabla u(t_0)\|_{H^k} + \|f\|_{LE^{*,k}[t_0,\infty)}), \qquad t_0 \geq 0, \ 
k \geq 0
\label{le}\end{equation}
in $\R \times \R^3$.
\end{definition}

The first local energy decay estimates for the wave equation were
proved in the work of Morawetz~\cite{MR0204828}, \cite{MR0234136},
\cite{MR0372432}; estimates of this type are often called Morawetz
estimates. There is by now an extensive literature devoted to this
topic and its applications; without being exhaustive we mention
\cite{Strauss}, \cite{KPV}, \cite{smithsogge}, \cite{MR1771575},
\cite{MR1945285}, \cite{MR2106340},
\cite{MR2217314}, \cite{MR2299569}, \cite{MR2128434}, \cite{MR2154375}.

The sharp form of the estimates as well as the notations above are
from Metcalfe-Tataru~\cite{MT}; this paper also contains a proof of
the local energy decay estimates for small (time dependent) long range
perturbations of the Minkowski space-time and further references. See
also \cite{MT1}, \cite{MR2266993}, \cite{MR2217314} for time dependent
perturbations, as well as, e.g., \cite{MR2001179}, \cite{BonyHafner}, \cite{SoggeWang} for time indepedent,
nontrapping perturbations.  There is a related
family of local energy estimates for the Schr\"odinger equation.  See,
e.g., the original works \cite{MR928265}, \cite{Sjo}, \cite{Veg} in
this direction as well as \cite{MR1795567}, \cite{CKS} in variable
geometries.  For notations and estimates most reminiscent to those
used here, we refer the reader to \cite{gS} and \cite{MMT}.

In Case B an estimate such as \eqref{le} cannot hold due to
the existence of trapped rays, i.e. null geodesics confined to a
compact spatial region. However a weaker form of the local energy decay
may still hold if the trapped null geodesic are hyperbolic. This is the
case for both the Schwarzschild metric  and for the
Kerr metric with angular momentum $|a| < M$.
To state such bounds we introduce a weaker version of the local energy
decay norm
\[
\begin{split}
  \| u\|_{LE^1_{w}} &= \| (1-\chi) \nabla u\|_{LE} + \| \la r\ra^{-1} u\|_{LE} \\
  \| u\|_{LE^1_{w}[t_0, t_1]} &= \| (1-\chi) \nabla u\|_{LE[t_0, t_1]} + \| \la r\ra^{-1} u\|_{LE[t_0, t_1]}
\end{split}
\]
for some spatial cutoff function $\chi$ which is smooth and compactly
supported. Heuristically, $\chi$ is chosen so that it equals $1$ in a
neighbourhood of the trapped set. We define as well a dual type norm
\[
\begin{split}
 \| f\|_{LE^*_w} &= \| \chi f\|_{L^2 H^1}+ \| (1-\chi) f\|_{LE^*} \\
\| f\|_{LE^*_w[t_0, t_1]} &= \| \chi f\|_{L^2[t_0, t_1] H^1}+ \| (1-\chi) f\|_{LE^*[t_0, t_1]}.
\end{split}
\]
As before we define the higher norms $LE^{1,k}_{w}$ respectively
$LE^{*,k}_w$.

\begin{definition}
 We say that the evolution \eqref{box} has the weak local energy decay
property if the following estimate holds:
\begin{equation}
 \| u\|_{LE^{1,k}_w[t_0,\infty)} \leq c_k (\|\nabla u(t_0)\|_{H^k} + \|f\|_{LE^{*,k}_w[t_0,\infty)} ),
 \qquad k \geq 0, \ t_0 \geq 0
\label{lew}\end{equation}
in either $\R \times \R^3$ or
in the exterior domain case.
\label{d:weakle}\end{definition}

Note that this implies in particular 
\begin{equation}\label{derivloss}
  \| u\|_{LE^{1,k}[t_0,\infty)} \leq c_k (\|\nabla u(t_0)\|_{ H^{k+1}} + \|f\|_{LE^{*,k+1}[t_0,\infty} ),
\end{equation}
hence we can get rid of the loss at the trapped set by paying the
price of one extra derivative.

Two examples where weak local energy decay is known to hold are the
Schwarz\-schild space-time and the Kerr space-time with small angular
momentum $|a| \ll M$.  In the Schwarzschild case the above form of the
local energy bounds with $k=0$ was obtained in \cite{MMTT}, following earlier
results in \cite{MR1732864}, \cite{MR1972492},
\cite{MR2113761},\cite{BS}, \cite{MR2259204},
\cite{MR2527808},~\cite{DR1}.  The number of derivatives lost in
\eqref{derivloss} can be improved to any $\epsilon > 0$ (see, for
example, \cite{BS}, \cite{MMTT}), but that is not relevant for the
problem at hand.

In the case of the Kerr space-time with small angular momentum $|a|
\ll M$ the local energy estimates were first proved in
Tataru-Tohaneanu~\cite{TT}, in a form which is compatible with
Definition~\ref{d:weakle}.  Stronger bounds near the trapped set as
well as Strichartz estimates are contained in the paper of
Tohaneanu~\cite{Mihai}. For related work we also refer the
reader to \cite{Mihai}, \cite{DR4} and \cite{AB}. For large
angular momentum $|a|<M$ a similar estimate was proved
in \cite{DR6} for axisymmetric solutions.

The high frequency analysis of the dynamics near the hyperbolic
trapped orbits is a very interesting related topic, but does not have
much to do with the present article.  For more information we refer
the reader to \cite{MR1294470}, \cite{MR2450154}, \cite{NZ},
\cite{WZ}, \cite{MR2051618}.

One disadvantage of the bound \eqref{derivloss} is that it is not very
stable with respect to perturbations. To compensate for that, for the
present result we need to introduce an additional local energy type
bound:

\begin{definition}\label{Enbdsdef}
  We say that the problem \eqref{box} satisfies stationary local
  energy decay bounds if on any time interval $[t_0,t_1]$ and $k \geq 0$ 
 we have
\begin{equation}\label{sle}
\| u\|_{LE^{1,k}[t_0,t_1]} \lesssim_k \|\nabla u(t_0)\|_{H^k} + \|\nabla u(t_1)\|_{H^k}+ 
\| f\|_{LE^{*,k}[t_0,t_1]} + \|\partial_t u\|_{LE^{0, k}[t_0,t_1]}.
\end{equation}
\end{definition} 

Unlike the weak local energy decay, here there is no loss of derivatives. Instead,
the price we pay is the local energy of $\partial_t u$ on the right. Heuristically,
\eqref{sle} can be viewed as a consequence of \eqref{lew} whenever $\partial_t$
is timelike near the trapped set. In the stationary case, this can be thought 
of as a substitute of an elliptic estimate at zero frequency.

While one can view the stationary local energy decay as a consequence
of the local energy decay, it is in effect far more robust and easier
to prove than the weak local energy decay provided that $\partial_t$
is timelike near the trapped set. This allows one to completely
sidestep all trapping related issues.  This difference is quite
apparent in the last section of the paper, where we separately establish 
both stationary local energy decay and weak local energy decay
for small perturbations of Kerr. While the former requires merely  smallness
of the perturbation uniformly in time, the latter needs a much stronger $t^{-1-}$ 
decay to Kerr.
 
Because of the above considerations, for our first (and main) result
in the theorem below we are only using as hypothesis the stationary
local energy condition.

\subsection{The main result}
 
Given a multiindex $\alpha$ we denote by $u_\alpha = Z^\alpha u$. By
$u_{\leq m}$ we denote the collection of all $u_\alpha$ with $|\alpha|
\leq m$.  We are now ready to state the main result of the paper:

 \begin{theorem}\label{main}
   Let $g$ be a metric which satisfies the conditions (i), (ii) in $\R
   \times \R^3$, or (i), (ii), (iii) in $\R \times \R^3\setminus
   B(0,R_0)$, and $V$ belonging to $S(r^{-3})$. Assume that the
   evolution \eqref{box} satisfies the stationary local energy bounds
   from Definition \ref{Enbdsdef}.  Suppose that in normalized
   coordinates the function $u$ solves $\Box u = f$ and is supported
   in the forward cone $C = \{ t \geq r - R_1\}$ for some $R_1>0$. Then the following
   estimate holds in normalized coordinates for large enough $m$:
\begin{equation}
| u(t,x)| \lesssim \kappa \frac{1}{\la t\ra \la t-r\ra^{2}}, \qquad | \nabla u(t,x)| 
\lesssim \kappa \frac{1}{\la r \ra \la t-r\ra^{3}}
\label{pointwiseest1}
\end{equation}
where 
\[
\kappa = \| u_{\leq m}\|_{LE^1} + \| t^\frac52 f_{\leq m}\|_{LE^*} + \|r t^\frac52 \nabla f_{\leq m}\|_{LE^*}.
\]
If in addition the weak local energy bounds \eqref{lew} hold then the same result 
is valid for all forward solutions to \eqref{box} with data $(u_0,u_1)$ and $f$ 
supported inside the cone $C$ and 
\[
\kappa = \| \nabla u(0)\|_{ H^m} + \| t^\frac52 f_{\leq m}\|_{LE^*} + \|r t^\frac52 \nabla f_{\leq m}\|_{LE^*}.
\]
  \end{theorem} 

 We remark that we actually prove a slightly stronger result, with $\kappa$ replaced by $C_4$ in Lemma~\ref{finaldecay}. 
   
  As an application of this result, in the last part of the paper we
  prove Price's Law for certain small, time-dependent perturbations of
  Kerr spacetimes with small angular momentum (and $V=0$).

  The problem of obtaining pointwise decay rates for linear and
  nonlinear wave equations has had a long history. Dispersive $L^1 \to
  L^\infty$ estimates providing $t^{-1}$ decay of $3+1$ dimensional
  waves in the Minkowski setting have been known for a long
  time. 

The need for weighted decay inside the cone arose in John's proof
\cite{MR535704} of the Strauss conjecture in $3+1$ dimensions.
Decay bounds for $\Box+V$ with $V= O(r^{-3})$, similar to those
given by Price's heuristics,  were obtained by Strauss-Tsutaya~\cite{MR1432072}
and Szpak~\cite{MR2475479}, \cite{Szpak}.  See also
Szpak-Bizo{\'n}-Chmaj-Rostworowski~\cite{MR2512504}.

A more robust way of proving pointwise estimates via $L^2$ bounds and
Sobolev inequalities was introduced in the work of Klainerman, who
developed the so-called vector field method, see for instance \cite{MR784477},
\cite{MR865359}. This idea turned out to have a myriad of
applications and played a key role in the
Christodoulou-Klainerman~\cite{MR1316662} proof of the asymptotic
stability of the Minkowski space time for the vacuum Einstein
equations.

In the context of the Schwarzschild space-time, Price was the first to
heuristically compute the $t^{-3}$ decay rate for linear waves. More
precise heuristic computations were carried out later by
Ching-Leung-Suen-Young~\cite{PhysRevLett.74.2414},
\cite{PhysRevLett.74.4588}. Following work of Wald~\cite{MR534342},
the first rigorous proof of the boundedness of the solutions to the
wave equation was given in Kay-Wald~\cite{MR895907}.

Uniform pointwise $t^{-1}$ decay estimates were obtained by
Blue-Sterbenz~\cite{BS} and also Dafermos-Rodnianski~\cite{MR2527808};
the bounds in the latter paper are stronger in that they extend
uniformly up to the event horizon. A local $t^{-\frac32}$ decay result
was obtained by Luk~\cite{L}. These results are obtained using
multiplier techniques, related to Klainerman's vector field method; in
particular the conformal multiplier plays a key role.

Another venue which was explored was to use the spherical symmetry in
order to produce an expansion into spherical modes and to study the
corresponding ode. This was pursued by Kronthaler~\cite{MR2226325},
\cite{K}, who in the latter article was able to establish the sharp
Price's Law in the spherically symmetric case. A related analysis was
carried out later by Donninger-Schlag-Soffer~\cite{DSS} for all the
spherical modes; they establish a $t^{-2-2k}$ local decay for the $k$-th
spherical mode. Later the same authors obtain the sharp $t^{-3-2k}$ in
\cite{DSS2}.

Switching to Kerr, the first decay results there were obtained by
Finster-Kamran-Smoller-Yau \cite{MR2215614}. Later
Dafermos-Rodnianski~\cite{DR4} and Anderson-Blue~\cite{AB} were able
to extend their Schwarzschild results to Kerr, obtaining almost a
$t^{-1}$ decay. This was improved to $t^{-1}$ by
Dafermos-Rodnianski~\cite{DR5} and to $t^{-\frac32}$ by
Luk~\cite{Luk2}.

Finally, the $t^{-3}$ decay result (Price's Law) was proved by
the second author in \cite{Tat}; the result there applies for a large
class of stationary, asymptotically flat space-times. In addition, in \cite{Tat} the
optimal decay is obtained  in the full forward light cone.

\section{ Vector fields and local energy decay}

The primary goal of this section is to develop localized energy
estimates when the vector fields $Z$ are applied to the solution $u$.

\subsection{Vector fields: notations and definitions}

For a triplet $\alpha = (i,j,k)$ of multi-indices $i$, $j$ and
nonnegative integer $k$ we denote
$|\alpha| = |i|+3|j|+9k$ and
\[
u_\alpha = T^i \Omega^j S^k u.
\]
On the family of such triplets we introduce the ordering induced by
the ordering of integers. Namely, if $\alpha_1 = (i_1,j_1,k_1)$ and
$\alpha_2 = (i_2,j_2,k_2)$ then we define
\[
\alpha_1 \leq \alpha_2 \qquad\equiv \qquad |i_1| \leq |i_2|, \ |j_1| \leq
|j_2|, \ k_1 \leq k_2.
\]
We use $<$ for the case when equality does not hold.  For an integer
$m$ we denote
\[
\alpha_1 \leq \alpha_2 + m \qquad\equiv \qquad\alpha_1 \leq \alpha_2 +
\beta, \quad |\beta| \leq m.
\]

We also define
\[
\begin{split}
u_{\leq m} &= (u_{\alpha})_{|\alpha|\le m} \\
u_{\leq \beta} &= (u_{\alpha})_{\alpha\le \beta}.
\end{split}
\]
and the analogues for $<$ instead of $\leq$.

We now study the commutation properties of the vector fields with $P$.
Denoting by $\Qsr$ the class of all operators of the form
\[\partial_\alpha g^{\alpha \beta}\partial_\beta + V,
\qquad g_{sr}^{\alpha \beta} \in S^Z(r^{-2}),\quad V \in S^Z(r^{-3}),\]
we see that $Q$ defined in \eqref{can-p-sr} consists of $R\Omega^2$
where $R\in S^Z_{rad}(r^{-3})$ plus an element of $\Qsr$.

We now record the commutators of $P$ with vector fields. The commutator with $T$  yields
\begin{equation}
[P,T] = Q, \qquad Q \in \Qsr.
\label{ptcom}\end{equation}
The same applies with $\Omega$,
\begin{equation}
[P,\Omega] = Q, \qquad Q \in \Qsr.
\label{pocom}\end{equation}
However, in the case of $S$ we get an extra contribution arising from the 
long range part of $P$,
\begin{equation}
PS- (S+2)P = Q + R \Omega^2, \qquad Q \in \Qsr, \qquad R\in S^Z_{rad}(r^{-3}).
\label{pscom}\end{equation}
Further commutations preserve the $\Qsr$ class. 
Thus we can write the equation for $\Omega^j S^k u $ in the form
\begin{equation}
P \Omega^j S^k u =  \Omega^j (S+2)^k Pu + Q u_{\leq 3j+9k-3} \qquad Q \in \Qsr.
\label{poscom}\end{equation}

Suppose the function $u$ solves the equation
\[
Pu = f.
\]
Commuting with all vector fields, we obtain equations for the functions
$u_\alpha$. These can be written in the forms
 \begin{equation}
   P u_\alpha = Q u_{< |\alpha|} + f_{\le |\alpha|}=: F_\alpha, \qquad Q \in \Qsr,
\label{pua} \end{equation}
 \begin{equation}
   \Box u_\alpha = Q u_{\leq |\alpha|} + R\Omega^2 u_{\le |\alpha|} +
   f_{\le |\alpha|}=: G_\alpha, \qquad Q \in \Qsr, \quad R\in S^Z_{rad}(r^{-3}).
\label{boxua} \end{equation}
For $F_{\alpha}$ and $G_\alpha$ we have pointwise bounds of the form
\begin{equation}\label{Fbound}
  |F_\alpha| \lesssim \frac{1}{\la r\ra ^3}(|\Omega^2 u_{<|\alpha|-6}| + 
  |u_{<|\alpha|}|) + \frac{1}{\la r\ra ^2}(|\nabla^2 u_{<|\alpha|}| + 
  |\nabla u_{<|\alpha|}|) + |f_{\le|\alpha|}|,
\end{equation}
\begin{equation}\label{Gbound}
  |G_\alpha| \lesssim \frac{1}{\la r\ra ^3}(|\Omega^2 u_{\leq |\alpha|}| + 
  |u_{\leq |\alpha|}|) + \frac{1}{\la r\ra ^2}(|\nabla^2 u_{\leq |\alpha|}| + 
  |\nabla u_{\leq |\alpha|}|) + |f_{\le |\alpha|}|.
\end{equation}

As a general principle, we will use the latter equation to improve the
bounds on $u_\alpha$ away from $r = 0$ (precisely for $r \gtrsim t$),
and the former near $r = 0$ (precisely for $r \ll t$).

\subsection{ The weak local energy decay}

The statement of the weak local energy decay property
in Definition~\ref{d:weakle}  includes the vector fields $T$ but not $S$ 
or $\Omega$. We remedy this in the following

\begin{lemma} 
Assume that the weak local local energy decay property \eqref{lew} holds.
Then we also have
\begin{equation}
\| u_{\leq m}\|_{LE^1} \lesssim \| \nabla u_{\leq m+1}(0)\|_{L^2}+ \|f_{\leq m+1}\|_{LE^*}.
\label{lew:vf}
\end{equation}
\label{l:levf}\end{lemma}
\begin{proof}
We use induction with respect to $m$. For $m=0$ the bounds \eqref{derivloss} and 
\eqref{lew:vf} coincide. Consider now some $m > 0$, and $\alpha$ a multiindex
with $|\alpha| = m$.  If $Z^\alpha$ contains only $T$ derivatives then the bound
for $u_\alpha$ follows directly from \eqref{derivloss}. Else we factor
\[
Z^\alpha = T^i \Omega^j S^k.
\]
Applying \eqref{derivloss} to $\Omega^j S^k u$ and using \eqref{poscom} we obtain
\[
\begin{split}
\|Z^\alpha u\|_{LE^1} \lesssim & \ 
\| \Omega^j S^k u\|_{LE^{i,1}} 
\\ \lesssim & \ \|\nabla \Omega^j S^k u(0)\|_{H^{i+1}}+ \|\Omega^j S^k f\|_{LE^{i+1,*}}+
  \| Q u_{\leq 3j+9k-3}\|_{LE^{i+1,*}}
\\
\lesssim & \ \|\nabla u_{\leq m+1}(0)\|_{L^2}+ \| f_{\leq
  m+1}\|_{LE^*} + \|\la r \ra^{-2} \nabla u_{\leq m-1}\|_{LE^{*}} 
\\ &\qquad\qquad\qquad\qquad\qquad\qquad\qquad\qquad\qquad\qquad
 + 
\|\la r \ra^{-3}  u_{\leq m-1}\|_{LE^{*}}
\\
\lesssim & \ \|\nabla u_{\leq m+1}(0)\|_{ L^2}+ \| f_{\leq m+1}\|_{LE^*} + \|u_{\leq m-1}\|_{LE^1} 
\end{split}
\]
which concludes our induction.
 \end{proof}

\subsection{The stationary local energy decay}

Our first aim here is to   include the vector fields  $S$ and $\Omega$
in the stationary local energy decay bounds. A second aim is to derive 
a variation of the same bounds with different weights.

\begin{lemma} 
Assume that the stationary local energy decay property \eqref{sle} holds.
Then for all $0 \leq  t_0 < t_1$ we also have
\begin{equation}
\| u_{\leq m}\|_{LE^1[t_0,t_1]} \lesssim 
\sum_{i=0,1} \|\nabla  u_{\leq m}(t_i)\|_{L^2}
+ \|f_{\leq m}\|_{LE^*} + \| \partial_t u_{\leq m}\|_{LE[t_0,t_1]},
\label{sle:vf}
\end{equation}
respectively 
\begin{equation}
\| \nabla u_{\leq m}\|_{L^2} \lesssim \sum_{i=0,1}
\| \la r \ra^\frac12 \nabla u_{\leq m}(t_i)\|_{L^2}
+ \|\la r \ra f_{\leq m}\|_{L^2} + \| \partial_t u_{\leq m}\|_{L^2}.
\label{sle:vfb}
\end{equation}

\end{lemma}
\begin{proof}
  The proof of \eqref{sle:vf} is identical to the proof of
  Lemma~\ref{l:levf} and is omitted.  We now prove \eqref{sle:vfb}. We
  begin with the case $m=0$, where we apply the classical method due
  to Morawetz.  Assume first that we are in Case A. 
Multiplying the equation $Pu = f$ by $(x \partial_x +
  1)u$ and integrating by parts we obtain
\[
\begin{split}
\int_{t_0}^{t_1} \int_{\R^3}Pu \cdot (x \partial_x + 1)u dx dt= & \ 
\int_{t_0}^{t_1}\int_{\R^3} |\nabla u|^2 + O(\la r \ra^{-1})  |\nabla u|^2 + O(\la r \ra^{-3}) |u|^2  
dx dt 
\\ & + \left.\int_{\R^3} O(\la r \ra) |\nabla u|^2 +  O(\la r \ra^{-1}) |u|^2 dx \right|_{t=t_0}^{t=t_1}.
\end{split}
\]
Using Cauchy-Schwarz on the left and estimating the $|u|^2$ terms by $|\nabla u|^2$
terms via Hardy inequalities we are left with 
\[
LHS\eqref{sle:vfb}(m=0) \lesssim RHS\eqref{sle:vfb}(m=0) + \| \nabla
u\|_{L^2(r \leq R)} + \|u\|_{L^2(r\leq R)}
\]
for some fixed large $R$.  Here the extraneous terms on the right are
only measured for small $r$, as the large $r$ contribution can be
absorbed on the left. However, to bound them for small $r$ we have at
our disposal the bound \eqref{sle:vf}, whose right hand side is
smaller than the right hand side of \eqref{sle:vfb}. The same outcome
is reached in Case B by inserting a cutoff function selecting the region $\{r \gg
1\}$ in the above computation.

To prove \eqref{sle:vfb} we can use a simpler direct argument since
there is no loss of derivatives on one hand and since we already have the
bound \eqref{sle:vf} to use to estimate $u_{<m}$ inside a compact
set. Precisely, for $|\alpha| = m$ we have
\[
P u_\alpha  = f_{\le m} + O(\la r \ra^{-2}) \nabla u_{\leq m} 
+ O(\la r \ra^{-3}) u_{\leq m}.
\]
Then we apply the $m=0$ case of \eqref{sle:vfb} for $u_\alpha$ and sum 
over $|\alpha| \leq m$. We obtain
\[
LHS\eqref{sle:vfb} \lesssim RHS\eqref{sle:vfb} + \| \nabla u_{\leq
  m}\|_{L^2(r \leq R)} + \|u_{\leq m}\|_{L^2(r\leq R)}
\]
and the last terms on the right are estimated by \eqref{sle:vf}.
\end{proof}

\section{The pointwise decay}

The strategy of the proof of our pointwise decay estimates is to iteratively 
improve the estimates via a two step approach.  The two steps are as follows:

(i) Use the properties of the fundamental solution for the constant
coefficient d'Alembertian $\Box$ via the equation \eqref{boxua}. This
yields improved bounds for $r \gg 1$, but no improvement at all 
for $r \sim 1$.

(ii) Use the stationary local energy decay estimates for the operator 
$P$ in the region $r \ll t$. This allows us to obtain improved bounds 
for small $r$. The transition from $L^2$ to pointwise bounds is done
in a standard manner via Sobolev type estimates.

\subsection{ The cone decomposition and Sobolev 
embeddings}

For the forward cone $C = \{ r \leq t\}$ we consider  a dyadic decomposition
 in time into sets
\[
C_{T} = \{ T \leq t \leq 2T, \ \ r \leq t\}.
\]
For each $C_T$ we need a further double dyadic decomposition of it
with respect to either the size of $t-r$ or the size of $r$, depending
on whether we are close or far from the cone,
\[
C_{T} = \bigcup_{1\leq  R \leq T/4}  C_{T}^{R}  \cup \bigcup_{1\leq  U < T/4} C_T^{U}
\]
where for $R,U > 1$ we set
\[
 C_{T}^{R} = C_T \cap \{ R < r < 2R \},
\qquad
C_{T}^{U} = C_T \cap \{ U < t-r < 2U\}
\]
while for $R=1$ and $U= 1$ we have
\[
 C_{T}^{R=1} = C_T \cap \{ 0 < r < 2 \},
\qquad
C_{T}^{U=1} = C_T \cap \{ 0 < t-r < 2\}
\]
with the obvious change for $C_T^1$ in Case B.  By $\tilde C_{T}^{R}$
and $\tilde C_{T}^{U}$ we denote enlargements of these sets in both
space and time on their respective scales. We also define
\[ 
C_{T}^{<T/2} = \bigcup_{R < T/4} C_T^R.
\]

The sets $C_T^R$ and $C_T^U$ represent the setting in which we apply
Sobolev embeddings, which allow us to obtain pointwise bounds from
$L^2$ bounds. Precisely, we have

\begin{lemma} \label{l:l2tolinf}
 For all $T \geq 1$ and $1 \leq R,U \leq T/4$  we have
 \begin{equation}
  \!  \| w\|_{L^\infty(C_T^{R})} \lesssim 
\frac{1}{T^{\frac12} R^{\frac32}} \sum_{i+j \leq 2} 
    \|S^i \Omega^j w\|_{L^2( \tilde C_T^{R})} +  \frac{1}{T^{\frac12} R^{\frac12}}
\sum_{i+j \leq 2}     \|S^i \Omega^j \nabla w\|_{L^2( \tilde C_T^{R})},
    \label{l2tolinf-r}\end{equation}
respectively
  \begin{equation}
    \| w\|_{L^\infty(C_T^{U})} \lesssim \frac{1}{T^{\frac32} U^{\frac12}} \sum_{i+j \leq 2} 
    \|S^i \Omega^j w\|_{L^2( \tilde C_T^{U})} +  \frac{U^{\frac12}}{T^{\frac32}}
 \sum_{i+j \leq 2}     \|S^i \Omega^j \nabla w\|_{L^2( \tilde C_T^{U})}.
    \label{l2tolinf-u}\end{equation}
\end{lemma}
\begin{proof}
  In exponential coordinates $(s,\rho,\omega)$ with $t = e^s$ and $r =
  e^{s+\rho}$ the bound \eqref{l2tolinf-r} is nothing but the usual
  Sobolev embedding applied uniformly in regions of size one.  The
  same applies for \eqref{l2tolinf-u} in exponential coordinates
  $(s,\rho,\omega)$ with $t = e^s$ and $t-r = e^{s+\rho}$.
\end{proof}

Expressed in terms of the local energy norm, the estimate 
\eqref{l2tolinf-r} yields

\begin{corollary}
  We have
  \begin{equation}
    \| w\|_{L^\infty(C_T^{<T/2})} \lesssim T^{-\frac12} \sum_{i+j \leq 2} 
    \|S^i \Omega^j w\|_{LE^1( \tilde C_T^{<T/2})}. 
    \label{letolinf}\end{equation}
\label{l:letolinf}
\end{corollary}

\subsection{The one dimensional reduction}

A main method to obtain pointwise estimates for $u_\alpha$ is by using
the positivity of the fundamental solution to the wave equation in
$3+1$ dimensions and the standard one dimensional reduction. 

%

For solutions to \eqref{boxua} with
vanishing initial data, we may apply time translation invariance and 
assume without loss of generality that $G_\alpha$ is supported in
$C=\{r\le t\}.$  Then 
define
\begin{equation}\label{Hdef}
  H_\alpha (t,r)= \sum_0^2 \|\Omega^i G_\alpha (t, r\omega)\|_{L^2(\S^2)},
\end{equation}
for $G_\alpha$ as in \eqref{boxua}.
By the Sobolev embeddings on the sphere we know that $|G_\alpha|\leq H_\alpha$.
Let $v_\alpha$ be the radial solution to 
\begin{equation}\label{1dbox}
  \Box v_\alpha = H_\alpha, \qquad v_\alpha(0)=\partial_t v_\alpha(0) =0.
\end{equation}
Then we can compare
\[
|u_\alpha| \leq v_\alpha.
\]
 We can rewrite the radial three dimensional equation \eqref{1dbox} as
a one dimensional problem
\[
(\partial_t ^2 - \partial_r ^2)(rv_\alpha) = rH_\alpha
\]
which has the solution
\begin{equation*}
 (rv_\alpha)(t, r) = \frac{1}{2}\int_0^t\int_{|r-t+s|}^{r+t-s}\rho H_\alpha(s,\rho)d\rho ds.
\end{equation*}
Assuming that $H_\alpha$ is supported in the forward cone  $t \geq r$, this 
is rewritten as 
\begin{equation}\label{Hsol}
rv_\alpha(t,r) = \frac12 \int_{D_{tr}} \rho H_\alpha(s,\rho) ds d\rho
\end{equation}
where $D_{tr}$ is the rectangle 
\[
D_{tr}=\{  0 \leq s - \rho \leq t-r, \quad  t-r \leq s+\rho \leq t+r 
 \}.
\]
 
In order to handle the contribution from the intial data in both cases and the fact that \eqref{boxua}
only holds outside of the 
cylinder $\R^+ \times B(0,R_0)$ in Case B,  we modify the above
argument.  The one dimensional reduction is only used to improve 
the bounds on $u_\alpha$ for large $r$. Hence we can truncate 
the functions $u_\alpha$ outside a large ball using a cutoff function $\chi_{out}$
which is identically one for large $r$. Then we can redefine $G_\alpha$
as
\[
G_\alpha = \Box (\chi_{out} u_\alpha).
\]
With this choice for $G_\alpha$ the bound \eqref{Gbound} still holds.  
Moreover by truncating outside of a sufficiently large ball,
$\chi_{out}u_\alpha$ has vanishing initial data.
We can use the one dimensional reduction to obtain bounds
for $u_\alpha$ for large $r$, while for small $r$, we shall rely on
the Sobolev-type estimates of Lemma \ref{l:l2tolinf}.

\subsection{ An initial decay bound}

Here we combine the above one dimensional reduction with the local
energy bounds in order to obtain an initial pointwise decay estimate
for the functions $u_\alpha$. This has the form
\begin{lemma}
The following estimate holds:
\begin{equation}\label{firstdecestderiv}
  |u_{\alpha}|\lesssim \frac{\log \la t-r\ra}{\la r\ra \la t-r\ra^{\frac{1}{2}}}
(\| u_{\leq \alpha + n}\|_{LE^1} + \|\la r \ra f_{\leq \alpha + n}\|_{LE^*})
\end{equation}
\end{lemma}
Here and later in this section $n$  represents a  large constant
which does not depend on $\alpha$ but may increase from one subsection
to the next. For the above lemma we can take $n = 25$ for instance,
but later on it becomes tedious and not particularly illuminating to keep track of the 
exact value of $n$.  We shall do so similarly for the enlargements of
the cones $\tilde C^R_T$, $\tilde C^U_T$.  We shall not track, though
we note that only a finite number will ever be required, each
subsequent enlargement which is needed and shall allow the enlargement 
to change from line to line while maintaining the same notation.
 
\begin{proof}
We assume that $r\gg 1$.  For $r \lesssim 1$ we instead 
use directly the Sobolev type embedding \eqref{letolinf}.  
 
For large $r$, we bound $u_\alpha$ by
  $v_\alpha$ and the function $H_\alpha$, using \eqref{Gbound}, by
\[
\| \la r\ra^2 H_\alpha\|_{LE} + \| \la r\ra^2 S H_\alpha\|_{LE}
\lesssim \| u_{\leq \alpha + 25}\|_{LE^1} 
+ \|\la r \ra f_{\leq \alpha  + 25}\|_{LE^*}.
\]
Hence it remains to show that the solution $v_\alpha$ to \eqref{1dbox}
satisfies
\begin{equation}
|v_\alpha| \lesssim \frac{\log \la t-r\ra}{\la r\ra \la t-r\ra^{\frac{1}{2}}}
(\| \la r\ra^2 H_{\alpha}\|_{LE}+ \| \la r\ra^2 S H_{\alpha}\|_{LE}).
\label{h-to-v}\end{equation}


We now prove \eqref{h-to-v}.  The index $\alpha$ plays no role in it
so we drop it.  One can then estimate $|rv|$ as
\begin{equation}\label{west}
  |rv(t,r)| \lesssim  \int_{D_{tr}} \rho H(s,\rho) ds d\rho. 
\end{equation}
We assume that $r \sim t$, as there is no further gain for smaller $r$
in estimating the integral on the right. 

We partition the set $D_{tr}$ into a double dyadic manner
as 
\[
D_{tr} = \bigcup_{R \leq t}  D_{tr}^R, \quad D_{tr}^R = D_{tr}\cap \{R<r<2R\}
\]
and estimate the corresponding parts of the above integral.
We consider two cases:

(i) $R < (t-r)/8$. Then we need to use the information about $Su$.
For any $(s,\rho) \in C$, let $\gamma_{s,\rho}(\tau)$ be the
integral curve corresponding to the vector field $S$, parametrized by
time, satisfying $\gamma_{s,\rho}(0)=(s,\rho)$.  The fundamental theorem of
calculus combined with the Cauchy-Schwarz inequality gives 
\[
 |H(\gamma_{s,\rho}(0))|^2 \le \frac{1}{s}\int_0^s
|H(\gamma_{s,\rho}(\tau))|^2\,d\tau + \frac{1}{s}\int_0^s
|(SH)(\gamma_{s,\rho}(\tau))|^2\,d\tau.
\] 
We apply this for $(s,\rho) \in D_{tr}^R$ and integrate. In the region
$D_{tr}^R$ we have $\rho \sim R$ and $ |s- (t-r)| \lesssim R$;
therefore we obtain
\[
\int_{D_{tr}^R} |H|^2 ds d\rho \lesssim \frac{R}{t-r} \int_{B_R} | H|^2+ |SH|^2 ds d\rho 
\]
where 
\[
B_R = \{(s,\rho):\  R/8 < \rho < 8R\}.
\]
Hence by Cauchy-Schwarz we conclude
\[
\begin{split}
\int_{D_{tr}^R} \rho H  ds d\rho \lesssim & \ \frac{R^\frac52}{\la t-r \ra^{\frac12}}
(\|H\|_{L^2(B_R)} + \|SH\|_{L^2(B_R)}) \\  \lesssim & \  \frac{1}{\la t-r \ra^{\frac12}}
(\|\la r \ra^2  H\|_{LE} + \|\la r \ra^2 SH\|_{LE}).
\end{split}
\]
The logarithmic factor in \eqref{h-to-v} arises in the dyadic $R$ summation.
We note that in the $L^2(B_R)$ norm above $H$ is viewed as a two 
dimensional function, whereas the $LE$ norm applies to $H$ as a 
radial function in $3+1$ dimensions.

(ii) $(t-r)/ 8 < R < t$.  Then we neglect  $SH$ and simply use Cauchy-Schwarz,
\[
\int _{D_{tr}^R} \rho H d\rho ds \lesssim R^\frac32 (t-r)^\frac12  \|H\|_{L^2(B_R)}
\lesssim R^{-1} (t-r)^\frac12 \| \la r \ra^2 H\|_{LE}. 
 \]
The dyadic $R$ summation is again straightforward.
\end{proof}

\subsection{Improved $L^2$ gradient bounds}
The bounds obtained in the previous step for $u_\alpha$ apply as well
to $\nabla u_\alpha$.  However, $u_\alpha$ and $\nabla u_\alpha$ do
not play symmetrical roles in the expressions for $F_\alpha$ and
$G_\alpha$.  In particular, the weights that come with $\nabla
u_\alpha$ are worse than the ones that come with $u_\alpha$. Hence,
when we seek to reiterate and improve the initial pointwise bound
\eqref{firstdecestderiv} it pays to have better bounds for $\nabla
u_\alpha$.  This is the aim of this step in the proof.  Our dyadic
$L^2$ gradient bound is contained in the next lemma:

\begin{lemma}
For $1 \ll U,R \leq T/4$ we have
\begin{equation}
  \| \nabla w\|_{L^2(C_{T}^{R})} \lesssim R^{-1}
 \| w\|_{L^2(\tilde C_{T}^R)}
  + T^{-1} \|S w \|_{L^2(\tilde C_{T}^R)}
  + R \| P w\|_{L^2(\tilde C_{T}^R)}
\label{ctrdu}\end{equation}
respectively
\begin{equation}
  \| \nabla w\|_{L^2(C_{T}^U)} \lesssim U^{-1}
 (\| w\|_{L^2(\tilde C_{T}^U)}
  + \|S w\|_{L^2(\tilde C_{T}^U)})
  + T \| P w\|_{L^2(\tilde C_{T}^U)}
\label{ctudu}\end{equation} \label{l2du}
\end{lemma}

Applied to $u_\alpha$ this gives
\begin{corollary}
For $1 \ll U,R \leq T/4$ we have
\begin{equation}
  \| \nabla u_{\alpha}\|_{L^2(C_{T}^{R})} \lesssim R^{-1}
 \| u_{\leq \alpha +n } \|_{L^2(\tilde C_{T}^R)} 
  + R \| f_{\leq \alpha}\|_{L^2(\tilde C_{T}^R)}
\label{ctrdu:va}\end{equation}
respectively
\begin{equation}
  \| \nabla u_\alpha\|_{L^2(C_{T}^U)} \lesssim U^{-1}
 \| u_{\leq \alpha +n }\|_{L^2(\tilde C_{T}^U)}
  + T \|   f_{\leq \alpha}\|_{L^2(\tilde C_{T}^U)}
\label{ctudu:va}\end{equation}
\end{corollary}

Applied to $\nabla u_\alpha$ we also obtain

\begin{corollary}
For $1 \ll U,R \leq T/4$ we have
\begin{equation}
  \| \nabla^2 u_{\alpha}\|_{L^2(C_{T}^{R})} \lesssim R^{-1}
 \| \nabla u_{\leq \alpha +n } \|_{L^2(\tilde C_{T}^R)} 
  + R \| \nabla f_{\leq \alpha}\|_{L^2(\tilde C_{T}^R)}
\label{ctrdu:dva}\end{equation}
respectively
\begin{equation}
  \| \nabla^2 u_\alpha\|_{L^2(C_{T}^U)} \lesssim U^{-1}
 \| \nabla u_{\leq \alpha +n }\|_{L^2(\tilde C_{T}^U)}
  + T \| \nabla  f_{\leq \alpha}\|_{L^2(\tilde C_{T}^U)}
\label{ctudu:dva}\end{equation}
\end{corollary}

\begin{proof}[Proof of Lemma \ref{l2du}]
To keep the ideas clear we first prove the lemma with $P$ replaced by $\Box$.
We consider a cutoff function $\beta$ supported in 
$\tilde C_{T}^{R}$ which equals $1$ on $C_{T}^{R}$.
Integration by parts gives
\begin{equation}
\int \beta( |\nabla_x w|^2 - |\partial_t w|^2) dx dt
= \int \Box w \cdot \beta w dxdt -\frac12  \int (\Box \beta) w^2 dx dt.
\label{inp}\end{equation}
To estimate $\nabla w$ we use the pointwise inequality
\[
|\nabla w|^2 \le M \frac{1}{(t-r)^{2}} |Sw|^2 + \frac{t}{t-r} 
(|\nabla_x w|^2 - |\partial_t w|^2)
\]
which is valid inside the cone $C$ for a fixed large $M$.
Hence 
\[
\int \beta|\nabla w|^2 dx dt \lesssim \int \frac{1}{(t-r)^{2}} \beta |Sw|^2
 +  \frac{t}{t-r}|\Box \beta| w^2  + \frac{t}{t-r}
 \beta |\Box w| |w| dxdt
\]
where all weights have a fixed size in the support of $\beta$.
The function $\beta$ can be further chosen so that $|\Box \beta| \lesssim r^{-2}$.
Then the conclusion of the lemma follows by applying Cauchy-Schwarz
to the last term. The argument for $C_{T}^U$ is similar, with the only difference 
that now we have $ |\Box \beta| \lesssim t^{-1}(t-r)^{-1}$.

Now consider the above proof but with $\Box$ replaced by $P$. 
Then, given the form of $P$ in \eqref{P}, \eqref{can-p-sr}, the relation \eqref{inp} is modified
as follows:
\begin{equation*}
\begin{split}
\int \beta( |\nabla_x w|^2 - |\partial_t w|^2) dx dt
= & \  \int P w \cdot \beta w dxdt -\frac12  \int ((P+V) \beta) w^2 dx dt 
\\ & \ + \int O(r^{-1}) \beta |\nabla w|^2 dx dt.
\end{split}
\end{equation*}
The bound for $(P+V) \beta$ is similar to the bound for $\Box \beta$, and 
one can easily see that the last error term is harmless.  The proof of the lemma 
is concluded.
\end{proof}

\subsection{Improved $L^2$ bounds for small $r$.}

A very unsatisfactory feature of our first pointwise bound
\eqref{firstdecestderiv} is the $r^{-1}$ factor which is quite bad for
small $r$. Here we devise a mechanism which allows us to replace this
factor by $t^{-1}$. Our main bound is an $L^2$ local energy bound,
derived using the stationary local energy decay assumption. 
We have

\begin{proposition} \label{lv:smallr}
  Assume that the problem \eqref{box} satisfies stationary local
  energy decay bounds \eqref{sle}.  Then the following estimates hold:
  \begin{equation}
    \| u_{\leq m} \|_{LE^1(C_{T}^{<T/2})} \lesssim
T^{-1} \| \la r \ra  u_{\leq m+n}\|_{LE^1( C_{T}^{<T/2})}
+   \|f_{\leq m+n} \|_{LE^*( C_{T}^{<T/2})}. 
    \label{smallr:main}
  \end{equation}
\end{proposition}

\begin{proof} 
  We first observe that we can harmlessly truncate $u$ to
  $C_{T}^{<T/2}$.  We will make this assumption throughout. Applying
  the stationary local energy decay estimate \eqref{sle:vf} for $u_{\leq m}$
  we obtain
\[
 \begin{split}
    \|u_{\leq m} \|_{LE^1(C_T)} \lesssim & \     
\|\nabla u_{\leq m}(T)\|_{L^2} +
    \|\nabla u_{\leq m}(2T)\|_{L^2} \\ & \
 +\|f_{\leq m}\|_{LE^*(C_T)} + \|\partial_t u_{\leq m} \|_{LE(C_T)}.
  \end{split}
\]
 For the last term on the right, due to the support of $u$, we have
  \[
\begin{split}
\|\partial_t u_{\leq m} \|_{LE(C_T)} \lesssim &\ \frac{1}{T} (\|S u_{\leq m}\|_{LE(C_T)} +  
\| r \partial_r u_{\leq m} \|_{LE(C_T)})
\\ 
\lesssim &\ T^{-\frac12} (\| \la r \ra^{-1} S u_{\leq m} \|_{L^2} +  \| \nabla  u_{\leq m} \|_{L^2})
\\ 
\lesssim &\  T^{-\frac12}(\| \nabla S u_{\leq m} \|_{L^2} +  \| \nabla  u_{\leq m} \|_{L^2})
\end{split}
\] 
where a Hardy inequality was used in the last step.  Next we consider
the time boundary terms.  By the fundamental theorem of calculus and
the Cauchy-Schwarz inequality, for any $s > 0$ we have
  \begin{equation}
|\nabla u(T,x)|^2\lesssim \frac{1}{s}\int_0^{s} 
|(\nabla u)(\gamma_{T,x}(\tau))|^2\,d\tau + \frac{s}{T^2}\int_0^{s} |(S\nabla 
  u)(\gamma_{T,x}(\tau))|^2\,d\tau.
\label{vfromsv}\end{equation}
This holds uniformly with respect to $s$, so we have the freedom to choose 
$0 < s \leq T$ favourably depending on $x$.  Suppose we take $s=T$.
Integrating the above estimate over $|x|\le  T$ for this choice of $s$
  and applying a change of coordinates yields
  \[
\int_{C_T \cap \{t = T\}} |\nabla u_{\leq m}(T)|^2\,dx
    \lesssim T^{-1} \int_{C_T} |\nabla u_{\leq m} |^2 + |S \nabla u_{\leq m} |^2 \,dx\,dt.
  \]
 A similar bound holds for $\nabla u_{\le m}(2T)$. Combining the last three estimates we 
obtain
\begin{equation}
    \| u_{\leq m}\|_{LE^1(C_{T})} \lesssim
    \|f_{\leq m+n}\|_{LE^*( C_{T})} + 
T^{-\frac12} \| \nabla u_{\leq m+n}\|_{L^2( C_{T})}
    \label{smallr:a}
  \end{equation}
which brings us half of the way to the proof of \eqref{smallr:main}.

We now consider the second expression on the right above, and we estimate it
via the same argument as before, but using \eqref{sle:vfb} instead of 
\eqref{sle:vf}. Indeed,  \eqref{sle:vfb} yields
\[ 
\begin{split}
    \|\nabla u_{\leq m} \|_{L^2(C_T)} \lesssim &\  \|\la r \ra f_{\leq m} \|_{L^2(C_T)} +   
\|\la r \ra^\frac12 \nabla u_{\leq m}(T)\|_{L^2} +
    \|\la r \ra^\frac12 \nabla u_{\leq m}(2T)\|_{L^2} \\ & \ + 
\|\partial_t u_{\leq m} \|_{L^2(C_T)} .
  \end{split}
\]
The last term is controlled as above by 
\[
\begin{split}
\|\partial_t u_{\leq m} \|_{L^2(C_T)} \lesssim &\  T^{-1} ( \|S u_{\leq m} \|_{L^2(C_T)}+ 
\|r \nabla u_{\leq m} \|_{L^2(C_T)})
\\
\lesssim & T^{-\frac12} ( \|\la r \ra  S u_{\leq m} \|_{LE^1(C_T)}+ 
\|\la r \ra  u_{\leq m} \|_{LE^1(C_T)}).
\end{split}
\]
The initial and final terms are estimated by a more careful use of \eqref{vfromsv}.
Namely, we integrate  \eqref{vfromsv} with $s(x) = \la r \ra^\frac12 T^\frac12$.
This yields
 \[
\begin{split}
\int \la r \ra | \nabla u_{\leq m}(T)|^2\,dx
    \lesssim & \ T^{-\frac12} \int_{C_T} \la r \ra^\frac12 | \nabla u_{\leq m} |^2 
+  T^{-\frac32} \int_{C_T}  \la r \ra^\frac32 |S \nabla u_{\leq m} |^2 \,dx\,dt 
 \end{split}
\]
which implies that
\[
\| \la r \ra^\frac12 u_{\leq m}(T)\|_{L^2} \lesssim  
T^{-\frac14} \|\la r \ra^\frac14 \nabla u_{\leq m}\|_{L^2(C_T)}
+ T^{-\frac12} \|\la r \ra S u_{\leq m}\|_{LE^1(C_T)}.
\]
Combining the last three bounds leads to 
\[
\begin{split}
  \| \nabla u_{\leq m}\|_{L^2(C_{T})} \lesssim & \
    \|\la r\ra  f_{\leq m+n} \|_{L^2( C_{T})} + 
T^{-\frac12} \| \la r \ra  u_{\leq m+n}\|_{LE^1( C_{T})} \\ & \ 
 +T^{-\frac14} \|\la r \ra^\frac14 \nabla u_{\leq m}\|_{L^2(C_T)}.
\end{split}
\]
We can discard the last term on the right by absorbing it into the left hand side for $r \ll T$
and into the second right hand side term for $r \sim T$. Thus we obtain
\begin{equation}
    \| \nabla u_{\leq m}\|_{L^2(C_{T})} \lesssim
    \|\la r\ra  f_{\leq m+n}\|_{L^2( C_{T})} + 
T^{-\frac12} \| \la r \ra u_{\leq m+n}\|_{LE^1( C_{T})}.
  \label{smallr:b}
  \end{equation}
Finally, the bound \eqref{smallr:main} is obtained by combining
\eqref{smallr:a} and \eqref{smallr:b}.
\end{proof}

\subsection{Improved pointwise bounds for 
$u_\alpha$ for small $r$}

Here we use Sobolev type  inequalities to make the transition from $L^2$
local energy bounds to pointwise bounds for small $r$. Our main estimate 
has the form 
\begin{proposition}
We have
\begin{equation}
\begin{split}
\| u_{\leq m} \|_{L^\infty ( C_{T}^{<T/2})}
& \  +\| \la r \ra \nabla u_{\leq m} \|_{L^\infty ( C_{T}^{<T/2})}  \lesssim 
T^{-\frac32} \|  u_{\leq m+n} \|_{LE( \tilde C_{T}^{<T/2})} 
\\ & \ + T^{-\frac12} \Bigl(\|f_{\leq m+n}\|_{LE^*( \tilde C_{T}^{<T/2})}+ 
\|\la r \ra^2 \nabla f_{\leq m+n}\|_{LE( \tilde C_{T}^{<T/2})}\Bigr).
\end{split}
\label{l2toli}\end{equation}
\label{p:pointr}
\end{proposition}
\begin{proof}

  In view of the Sobolev bound \eqref{letolinf} in
  Corollary~\ref{l:letolinf}, the estimate \eqref{l2toli} would follow
  from
\begin{equation}
\begin{split}
\| u_{\leq m} \|_{LE^1 ( C_{T}^{<T/2})} & \ 
+\| \la r \ra \nabla u_{\leq m} \|_{LE^1 ( C_{T}^{<T/2})} \lesssim
T^{-1} \|  u_{\leq m+n} \|_{LE( \tilde C_{T}^{<T/2})} 
\\ & \ + \| f_{\leq m+n}\|_{LE^*(\tilde C_T^{<T/2})} + \| \la r \ra^2 \nabla f_{\leq
  m+n}\|_{LE(\tilde C_{T}^{<T/2})}.
\end{split}
\label{l2toli:a}\end{equation}
To prove this we begin with the local energy bound \eqref{smallr:main}
and add to it corresponding bounds for the second derivatives of
$u_{\leq m}$.  For that we use \eqref{ctrdu:dva}; applied uniformly in dyadic
regions $C_T^R$ with $R \lesssim T$ it yields
\[
\|\la r \ra \nabla^2 u_{\leq m}\|_{LE(C_{T}^{<T/2})} \lesssim 
\|  u_{\leq m+n}\|_{LE^1(C_{T}^{<T/2})} + \| \la r \ra^2 \nabla f_{\leq m+n}\|_{LE(C_{T}^{<T/2})}.
\]
Combined with \eqref{smallr:main}, this leads  to
\begin{equation}
\begin{split}
LHS\eqref{l2toli:a} \lesssim & \
T^{-1} \| \la r \ra u_{\leq m+n} \|_{LE^1( \tilde C_{T}^{<T/2})} 
+ \| f_{\leq m+n}\|_{LE^*(C_T^{<T/2})} \\ & \ + \| \la r \ra^2 \nabla f_{\leq m+n}\|_{LE(C_{T}^{<T/2})}.
\end{split}
\label{l2toli:b}\end{equation}

It remains to make the transition from \eqref{l2toli:b} to
\eqref{l2toli:a}. This is done via \eqref{ctrdu:va} applied uniformly
in dyadic regions $C_T^R$ with $R \lesssim T$, which yields
\[
\| \la r \ra u_{\leq m} \|_{LE^1( \tilde C_{T}^{<T/2})} \lesssim 
 \|  u_{\leq m+n} \|_{LE( \tilde C_{T}^{<T/2})} + 
\| \la r \ra^2 f_{\leq m+n} \|_{LE( \tilde C_{T}^{<T/2})}.
\]
Hence \eqref{l2toli:a} follows.
\end{proof}

\subsection{ Improved pointwise estimates for 
$\nabla u_\alpha$ near the cone}

Here we convert the improvement in the $L^2$ bounds for $\nabla
u_{\leq m}$ given by Lemma~\ref{l2du} near the cone into a similar
$L^\infty$ bound. This shows that the pointwise bounds for $\nabla
u_{\leq m}$ are better than those for $u_{\leq m}$ by a factor of $\la t-r \ra^{-1}$.
\begin{proposition}
We have
\begin{equation}
\begin{split}
U \| \nabla u_{\leq m} \|_{L^\infty ( C_{T}^{U})}
  \lesssim & \  \|  u_{\leq m+n} \|_{L^\infty( \tilde C_{T}^{U})} + 
 T^{-\frac12} U^{\frac12}  \| f_{\leq m+n}\|_{L^2( \tilde C_{T}^{U})}\\ & \ + 
T^{-\frac12} U^{\frac32}  \| \nabla f_{\leq m+n}\|_{L^2( \tilde C_{T}^{U})}.
\end{split}
\label{l2toli:u}\end{equation}
\label{p:pointu}
\end{proposition}
\begin{proof}
The proof is similar to the proof of Proposition~\ref{p:pointr}. 
In view of the Sobolev inequality \eqref{l2tolinf-u}, the bound \eqref{l2toli:u} would follow from
\begin{equation*}
\begin{split}
U \| \nabla u_{\leq m} \|_{L^2 ( C_{T}^{U})} + U^2 \| \nabla^2 u_{\leq m} \|_{L^2 ( C_{T}^{U})}
\lesssim & \  \|  u_{\leq m+n} \|_{L^2 ( C_{T}^{U})}  + 
 UT \| f_{\leq m+n}\|_{L^2 ( C_{T}^{U})}
\\ & \ + U^2T \| \nabla f_{\leq m+n}\|_{L^2 ( C_{T}^{U})}.
\end{split}
\end{equation*}
To prove the bound on $\nabla u_{\leq m}$ we use \eqref{ctudu:va}, 
which shows that 
\[
 U \| \nabla u_{\leq m} \|_{L^2 ( C_{T}^{U})}  \lesssim 
\|  u_{\leq m+n} \|_{L^2 ( \tilde C_{T}^{U})} + UT \| f _{\leq m+n} \|_{L^2 (\tilde C_{T}^{U})}. 
\]
To prove the bound on $\nabla^2 u_{\leq m}$ we use \eqref{ctudu:dva} 
to obtain
\[
 U^2 \| \nabla^2 u_{\leq m} \|_{L^2 ( C_{T}^{U})}  \lesssim 
 U \| \nabla u_{\leq m+n} \|_{L^2 ( \tilde C_{T}^{U})} 
+ U^2 T \| \nabla f _{\leq m+n} \|_{L^2 ( \tilde C_{T}^{U})}.
\]
\end{proof}

\subsection{ Gradient bounds associated to 
the first decay bound}

Here we start with the bound \eqref{firstdecestderiv} for $u_{\leq m}$
and improve it for small $r$, as well as complement it with gradient
bounds.

\begin{lemma}
The following pointwise estimates hold:
\begin{equation}
|u_{\leq m}| \lesssim C_1  \frac{\log \la t-r \ra }{t \la t-r \ra^\frac12}, \qquad 
|\nabla u_{\leq m}| \lesssim C_1  \frac{\log \la t-r \ra }{\la r \ra \la t-r \ra^\frac32}
\label{firstgrad}\end{equation}
where
\[
 C_1 = \!
\| u_{\leq m+n}\|_{LE^1} + \sup_{R,U} T^\frac12 R^\frac12 U^\frac12
 \| f_{\leq m+n}\|_{L^2(C_{T}^{R,U}) }+ T^{-\frac12} R^\frac32 U^\frac32
 \| \nabla f_{\leq m+n}\|_{L^2(C_T^{R,U})}. \!
\]
\end{lemma}

Here for brevity $C_T^{R,U}$ stands for either $C_T^R$ or $C_T^U$,
with the natural convention that $R \sim T$ in $C_T^U$ and $U \sim T$
in $C_T^R$. The proof is a direct application of
Propositions~\ref{p:pointr}, \ref{p:pointu}, using
\eqref{firstdecestderiv} as a starting point.

\subsection{The second decay bound}

\begin{lemma}
The following decay estimate holds:
\begin{equation}\label{scnddecest}
|u_{\leq m}| \lesssim   C_2  \frac{\log\la t-r\ra}{t \la t-r \ra}, \qquad 
|\nabla u_{\leq m}| \lesssim C_2  \frac{\log\la t-r\ra}{\la r \ra \la t-r \ra^2}
\end{equation}
where
\[
 C_2 = 
\| u_{\leq m+n}\|_{LE^1} + \sup_{R,U} T R^{\frac12} U^{\frac12} \| f_{\leq m+n}\|_{L^2(C_{T}^{R,U}) }+ 
 R^{\frac32} U^{\frac32} \|\nabla  f_{\leq m+n}\|_{L^2(C_T^{R,U})}.
\]
\end{lemma}

\begin{proof}
By the Sobolev embeddings of Lemma \ref{l:l2tolinf} we have 
\[
|f_{\leq m}| \lesssim \frac{1}{t \la r \ra^2 \la t-r \ra} C_2.
\]
Also by \eqref{firstgrad} we have 
\[
|u_{\leq m+6} | \lesssim   \frac{\log\la t-r\ra}{t \la t-r \ra^\frac12} C_2, 
\qquad |\nabla u_{\leq m+6} | \lesssim   \frac{\log\la t-r\ra}{\la r \ra \la t-r \ra^\frac32} C_2.
\]
Hence for the functions $G_{\leq m}$ we obtain
\[
|G_{\leq m}| \lesssim \frac{1}{ \la r \ra^3 \la t-r \ra} C_2.
\]
Computing via the one dimensional reduction, this leads to a bound
for $u_{\leq m}$ of the form
\[
|u_{\leq m}| \lesssim   C_2  \frac{\log\la t-r\ra}{\la r \ra  \la t-r \ra}
\]
 which is comparable to  \eqref{scnddecest} near the cone, but it is weaker
for $r \ll t$. Then the small $r$ bound for $u_{\leq m}$ and the bound for $\nabla u_{\leq m}$ are
obtained from Propositions~\ref{p:pointr}, \ref{p:pointu} (we need to
increase $n$ appropriately at this stage).
\end{proof}

\subsection{The third decay bound}

\begin{lemma}
The following decay estimate holds:
\begin{equation}\label{thrddecest}
|u_{\leq m}| \lesssim   C_3  \frac{\log^3\la t-r\ra}{t \la t-r \ra^2}, \qquad 
|\nabla u_{\leq m}| \lesssim C_3  \frac{\log^3\la t-r\ra}{\la r\ra \la t-r \ra^3}
\end{equation}
where
\[
 C_3 = 
\| u_{\leq m+n}\|_{LE^1} + \sup T^2 R^{\frac12} U^{\frac12} \| f_{\leq m+n}\|_{L^2(C_{T}^{R,U}) }
+ T R^{\frac32} U^{\frac32} \|\nabla f_{\leq m+n}\|_{L^2(C_T^{R,U})}. 
\]
\label{lem3}
\end{lemma}

\begin{proof}
As before, the main step in the proof is to obtain a pointwise bound for $u_{\leq m}$
which coincides with \eqref{thrddecest} near the cone,
\begin{equation}
 |u_{\leq m} | \lesssim  \frac{\log^3\la t-r\ra}{\la r \ra  \la t-r \ra^2} C_3.
\label{pointgoal}\end{equation}
Once this is done, the full estimate \eqref{thrddecest} follows
easily by a direct application of
Propositions~\ref{p:pointu},~\ref{p:pointr}.  However, at this stage
there is a new twist, namely that the one dimensional reduction no
longer suffices for the proof of \eqref{pointgoal}.  

The pointwise bound for $f_{\leq m}$ has the form
\[
|f_{\leq m}| \lesssim \frac{1}{t^2 \la r \ra^2 \la t-r \ra} C_3.
\]
Also by \eqref{scnddecest} we have 
\[
|u_{\leq m+6} | \lesssim   \frac{\log\la t-r\ra}{t \la t-r \ra} C_3, 
\qquad |\nabla u_{\leq m+6} | \lesssim   \frac{\log\la t-r\ra}{\la r \ra \la t-r \ra^2} C_3.
\]

We use  the wave equation for $u_{\leq m}$ given by \eqref{boxua}, and 
 rewrite $G_{\leq m}$ in the form
\[
G_{\leq m} = f_{\leq m} + a u_{\le m+6} +  \partial_t (b u_{<m+6}), 
\qquad a \in S^Z(r^{-3}), \ \ b \in S^Z(r^{-2}).
\]
Here we can confine ourselves to $\partial_t$ derivatives in the last term
because for any $S$ and $\Omega$ component we gain a factor of $r^{-1}$
and include it in the second term.
We split  $G_{\leq m}$ into two parts, 
\[
G_{\leq m} = G_{\leq m}^1+  G_{\leq m}^2
\]
with
\[
G_{\leq m}^1 = f_{\leq m} + a u_{\le m+6} +  \partial_t  ((1-\chi(t,r)) b u_{<m+6}),
\quad G_{\leq m}^2 =  \partial_t (\chi(t,r) b u_{<m+6})
\]
Here $\chi$ is a smooth cutoff selecting the region $t-r \ll t$.

The function $G_{\leq m}^1 $ contains the part of $G_{\leq m}$ which has
good pointwise bounds,
\[
|G_{\leq m}^1 | \lesssim  
\frac{\log\la t-r\ra}{t \la r \ra^3 \la t-r \ra} C_3.
\]
Computing via the one dimensional reduction, this gives 
the pointwise bound \eqref{pointgoal} for the corresponding part $u_{\leq m}^1$ of
$u_{\leq m}$.

Next we prove the same bound for the output  $u_{\leq m}^2$
of $G_{\leq m}^2$. Taking absolute values and applying the 
one dimensional reduction does not work, as it misses 
a cancellation due to the presence of the derivatives.
Instead we do a more precise computation

\begin{lemma}
Consider a smooth function $f$ supported in $ \{\frac{t}2  \leq  r  \leq t \}$
so that
\begin{equation}
|f|+|Sf|+|\Omega f|  + \la t-r \ra |\partial_r f|  \lesssim \frac{1}{t^3 \la t-r \ra \log^2\la t-r \ra}.
\label{fall}\end{equation}
Then the forward solution $u$  to 
\[
\Box u = \partial_t f
\]
satisfies the bound
\begin{equation}
| u| \lesssim \frac{1}{t \la t-r \ra^2}.
\label{fallu}\end{equation}
\label{l:maindr}
\end{lemma}
We note that if \eqref{fall} is replaced by 
\begin{equation}
|f|+|\nabla f| + |Sf|+|\Omega f|  + \la t-r \ra |\partial_r f|  \lesssim \frac{ \log\la t-r \ra}{t^3 \la t-r \ra}
\label{fall1}\end{equation}
then \eqref{fallu} is trivially replaced by 
\begin{equation}
| u| \lesssim \frac{\log^3\la t-r \ra}{t \la t-r \ra^2}
\label{fallu1}\end{equation}
which suffices to conclude the proof of Lemma~\ref{lem3}.

\begin{proof}
The function $u$ is expressed in the form $u = \partial_t v$
with $v$ the forward solution to $\Box v = f$.
Via the one dimensional reduction applied to $v$, $\nabla v$ 
$\Omega v$, $S v$ and $(t \partial_i
+ x_i \partial_t) v$ we obtain 
\[
|v|+|\nabla v| +|Sv|+|\Omega v|  + \sum_i  | (t \partial_i
+ x_i \partial_t) v|  \lesssim \frac{1}{t \la t-r \ra} 
\]
where the main contribution comes from the cone.
The above left hand side dominates $\la t-r\ra \partial_t v$;
therefore the proof of the lemma is complete.
\end{proof}

\end{proof}

\subsection{The fourth (and final) decay bound}

Here we reiterate one last time to remove the logarithms in
\eqref{thrddecest} and establish the final bound, which concludes the
proof of the Theorem.

\begin{lemma}\label{finaldecay}
  The following decay estimate holds:
  \begin{equation}\label{frthdecest}
    |u_{\leq m}| \lesssim   C_4  \frac{1}{t \la t-r \ra^2}, \qquad 
    |\nabla u_{\leq m}| \lesssim C_4  \frac{1}{\la r \ra \la t-r \ra^3}
  \end{equation}
  where
  \[
  C_4 = \| u_{\leq m+n}\|_{LE^1} + \sup_T \sum_{R,U} T^2
  R^{\frac12} U^{\frac12} \| f_{\leq m+n}\|_{L^2(C_{T}^{R,U}) } +
  T R^{\frac32} U^{\frac32} \|\nabla f_{\leq m+n}\|_{L^2(C_T^{R,U})}.
  \]
\end{lemma}

\begin{proof}
  The argument is similar to the previous step, but with some extra
  care in order to avoid the logarithmic losses. The main goal is
  again to obtain a pointwise bound for $u_{\leq m}$ which coincides
  with \eqref{thrddecest} near the cone,
  \begin{equation}
    |u_{\leq m} | \lesssim  \frac{1}{\la r \ra  \la t-r \ra^2} C_4
    \label{pg}\end{equation}
  after which the full estimate \eqref{thrddecest} follows from
  Propositions~\ref{p:pointu},~\ref{p:pointr}.

  The pointwise bound for $f_{\leq m}$ still has the form
  \[
  |f_{\leq m}| \lesssim \frac{1}{t^2 \la r \ra^2 \la t-r \ra} C_4,
  \]
  but now in addition we have dyadic summability with respect to $R$
  and $U$.  Also by \eqref{thrddecest} we have
  \[
  |u_{\leq m+6} | \lesssim \frac{\log^3\la t-r \ra}{t \la t-r
    \ra^2} C_4, \qquad |\nabla u_{\leq m+6} | \lesssim
  \frac{\log^3\la t-r \ra}{\la r \ra \la t-r \ra^3} C_4.
  \]
  We split $G_{\leq m} = G_{\leq m}^1+ G_{\leq m}^2$ exactly as before.

  For $G_{\leq m}^1$ we use the one dimensional reduction, based on the
  bound
  \[
  |G_{\leq m}^1| \lesssim |f_{\leq m}| + \frac{\log^3\la t-r \ra}{\la r
    \ra^3 \la t-r \ra^3} C_4,
  \]
  which gives the pointwise bound \eqref{pg} for $u_{\leq m}^1$.  The
  dyadic summability for $f$ causes the absence of logarithms in the
  bound for the contribution of $f$. The contribution of the second
  term is of the order of
  \[
  \frac{\log^3\la t-r \ra}{\la r\ra \la t-r \ra^4}
  \]
  where we have a full extra power of $\la t-r \ra$ available to
  control the logarithms.

  Finally, the contribution of $G_{\leq m}^2$ is controlled by
  Lemma~\ref{l:maindr}.
\end{proof}

\newsection{Perturbations of Kerr spacetimes}

We now present an application of Theorem~\ref{main} to general
relativity. We are able to recover Price's Law not only for
Schwarzschild and Kerr spacetimes, but also for certain small,
time-dependent perturbations thereof. We begin by presenting the
results obtained in \cite{MMTT} and \cite{TT} for the Schwarzschild
and Kerr metrics. We continue in \ref{ss:sle} with a proof of
stationary local energy decay estimates for perturbations of
Schwarzschild; while the perturbations are required to be small, no
decay to Schwarzschild is assumed.  This result applies as well to
small perturbations of Kerr with small angular momentum. Finally, in
\ref{ss:le} we establish weak local energy decay estimates for small
perturbations of Kerr; here, we essentially require a $t^{-1-}$ decay
rate of the perturbed metric to Kerr.

\subsection{ The Schwarzschild and Kerr metrics}

The Kerr metric in Boyer-Lindquist coordinates is given by
\[
ds^2 = g_{tt}dt^2 + g_{t\phi}dtd\phi + g_{rr}dr^2 +
g_{\phi\phi}d\phi^2 + g_{\theta\theta}d\theta^2
\]
where $t \in \R$, $r > 0$, $(\phi,\theta)$ are the spherical
coordinates on $\S^2$ and
\[
g_{tt}=-\frac{\Delta-a^2\sin^2\theta}{\rho^2}, \qquad
g_{t\phi}=-2a\frac{2Mr\sin^2\theta}{\rho^2}, \qquad
g_{rr}=\frac{\rho^2}{\Delta}
\]
\[ g_{\phi\phi}=\frac{(r^2+a^2)^2-a^2\Delta
  \sin^2\theta}{\rho^2}\sin^2\theta, \qquad g_{\theta\theta}={\rho^2}
\]
with
\[
\Delta=r^2-2Mr+a^2, \qquad \rho^2=r^2+a^2\cos^2\theta.
\]
Here $M$ represents the mass of the black hole, and $aM$ its angular
momentum.
 
A straightforward computation gives us the inverse of the metric:
\[ g^{tt}=-\frac{(r^2+a^2)^2-a^2\Delta\sin^2\theta}{\rho^2\Delta},
\qquad g^{t\phi}=-a\frac{2Mr}{\rho^2\Delta}, \qquad
g^{rr}=\frac{\Delta}{\rho^2},
\]
\[
g^{\phi\phi}=\frac{\Delta-a^2\sin^2\theta}{\rho^2\Delta\sin^2\theta} ,
\qquad g^{\theta\theta}=\frac{1}{\rho^2}.
\]

The case $a = 0$ corresponds to the Schwarzschild space-time.  One can
view $M$ as a scaling parameter, and $a$ scales in the same way as
$M$. Thus $M/a$ is a dimensionless parameter.  We shall subsequently
assume that $a$ is small $a/M \ll 1$, so that the Kerr metric is a
small perturbation of the Schwarzschild metric.  One could also set $M
= 1$ by scaling, but we prefer to keep $M$ in our formulas.  We let
$g_{\mathbf S}$, $g_{\mathbf K}$ denote the Schwarzschild, respectively
Kerr metric, and $\Box_{\mathbf S} $, $\Box_{\mathbf K} $ denote the
associated d'Alembertians.

In the above coordinates the Kerr metric has singularities at $r = 0$
on the equator $\theta = \pi/2$ and at the roots of $\Delta$, namely
$r_{\pm}=M\pm\sqrt{M^2-a^2}$. The singularity at $r=r_{+}$ is just a
coordinate singularity, and corresponds to the event horizon.  The
singularity at $r = r_-$ is also a coordinate singularity; for a
further discussion of its nature, which is not relevant for our
results, we refer the reader to \cite{Ch,HE}.  To remove the
singularities at $r = r_{\pm}$ we introduce functions $r^*$, $v_{+}$
and $\phi_{+}$ so that (see \cite{HE})
\[
dr^*=(r^2+a^2)\Delta^{-1}dr, \qquad dv_{+}=dt+dr^*, \qquad
d\phi_{+}=d\phi+a\Delta^{-1}dr.
\]
We call $v_{+}$ the advanced time coordinate.
The metric then takes the Eddington-Finkelstein form 
\[
\begin{split}
  ds^2= &\
  -(1-\frac{2Mr}{\rho^2})dv_{+}^2+2drdv_{+}-4a\rho^{-2}Mr\sin^2\theta
  dv_{+}d\phi_{+} -2a\sin^2\theta dr d\phi_{+}\\
  & \ +\rho^2 d\theta^2  +\rho^{-2}[(r^2+a^2)^2-\Delta a^2\sin^2\theta]
\sin^2\theta  d\phi_{+}^2
\end{split}
\]
which is smooth and nondegenerate across the event horizon up to but
not including $r = 0$.

In order to talk about perturbations of Kerr we need to settle on a
suitable coordinate frame. The Boyer-Lindquist coordinates are
convenient at spatial infinity but not near the event horizon while
the Eddington-Finkelstein coordinates are convenient near the event
horizon but not at spatial infinity. To combine the two
we replace the $(t,\phi)$ coordinates with $(\tv,\tphi)$ as follows.

As in \cite{MMTT} and \cite{TT}, we define
\[
\tv = v_{+} - \mu(r)
\]
where $\mu$ is a smooth function of $r$. In the $(\tv,r,\phi_{+},
\theta)$ coordinates the metric has the form
\[
\begin{split}
  ds^2= &\ (1-\frac{2Mr}{\rho^2}) d\tv^2
  +2\left(1-(1-\frac{2Mr}{\rho^2})\mu'(r)\right) d\tv dr \\
  &\ -4a\rho^{-2}Mr\sin^2\theta d\tv d\phi_{+} + \Bigl(2 \mu'(r) -
  (1-\frac{2Mr}{\rho^2}) (\mu'(r))^2\Bigr)  dr^2 \\
  &\ -2a\theta (1+2\rho^{-2}Mr\mu' (r))\sin^2dr d\phi_{+} +\rho^2
  d\theta^2 \\
  &\ +\rho^{-2}[(r^2+a^2)^2-\Delta a^2\sin^2\theta]\sin^2\theta
  d\phi_{+}^2.
\end{split}
\]
On the function $\mu$ we impose the following two conditions:

(i) $\mu (r) \geq \rs$ for $r > 2M$, with equality for $r > {5M}/2$.

(ii) The surfaces $\tv = const$ are space-like, i.e.
\[
\mu'(r) > 0, \qquad 2 - (1-\frac{2Mr}{\rho^2}) \mu'(r) > 0.
\]

For convenience we also introduce
\[
\tphi = \zeta(r)\phi_{+}+(1-\zeta(r))\phi
\]
where $\zeta$ is a cutoff function supported near the event horizon
and work in the $(\tv,r,\tphi, \theta)$ coordinates which are
identical to $(t,r,\phi,\theta)$ outside of a small neighborhood of
the event horizon.

Given $r_{-} < r_e <r_{+}$ we consider the Kerr metric and the
corresponding wave equation
\begin{equation}
  \Box_{\mathbf K}  u = f 
  \label{boxsinhom}\end{equation}
in the cylindrical region
\begin{equation}
  \M =  \{ \tv \geq 0, \ r \geq r_e \} 
  \label{mr}\end{equation}
with initial data on the space-like surface
\begin{equation}
  \Sigma^- =  \M \cap \{ \tv = 0 \}.
  \label{mr-}\end{equation}
The lateral boundary of $\M$,
\begin{equation}
  \Sigma^+ =   \M \cap \{ r = r_e\}, \qquad  \Sigma^+_{[\tv_0, \tv_1]} = \Sigma^+ \cap \{\tv_0 \leq \tv \leq 
  \tv_1\}
  \label{mr+}\end{equation}
is also space-like and can be thought of as the exit surface for all
waves which cross the event horizon. This places us in the context of 
Case B in section \ref{ss:cases}. The choice of $r_e$ is not important; for 
convenience one may simply use $r_e = M$ for all Kerr metrics 
with $a/M \ll 1$ and small perturbations thereof.

A main difficulty in proving local energy decay in Schwarzschild/Kerr
space-times is due to the presence of trapped rays (null
geodesics). In the Schwarzschild case, this occurs on the photon
sphere $\{r = 3M\}$. Consequently the local energy bounds have a loss
at $r = 3M$.  To localize there we use a smooth cutoff function
$\chi_{ps}(r)$ which is supported in a small neighborhood of $3M$ and which
equals $1$ near $3M$.  By $\tilde\chi_{ps}(r)$ we denote a smooth cutoff
that equals $1$ on the support of $\chi_{ps}$.  Then we define suitable
modifications of the $LE^1$, respectively $LE^*$ norms by
\[
\| u\|_{LE_{\mathbf S}^1} =  \| \partial_r u\|_{LE} +\|(1-\frac{3M}r)\nabla u\|_{LE}
+  
\|r^{-1} u\|_{LE}
\]
\[
\| f\|_{LE_{\mathbf S}^*} = \|(1-\chi_{ps}) f\|_{LE^*} + \| \chi_{ps} f\|_{H^1+ (1-\frac{3M}r)L^2} 
\]
With these notations, the local energy decay estimates
established in \cite{MMTT} have the form

\begin{theorem}\label{Swz}
  Let $u$ solve $\Box_{\mathbf S} u = f$ in $\M$. Then
  \begin{equation}
  \| u\|_{LE_{\mathbf S}^1}  + 
\sup_{\tilde v \geq 0} \| \nabla u(\tv)\|_{L^2}
    \lesssim\|\nabla  u(0)\|_{L^2} + \| f \|_{LE_{\mathbf S}^*}.
    \label{Sest}\end{equation}
\end{theorem}
As written there is a loss of one derivative at $r = 3M$. This can be 
improved to an $\epsilon$ loss, or even to a logarithmic loss, 
see \cite{MMTT}, but that is not so relevant for our purpose here.

In the case of Kerr, the trapped rays are no longer localized on a
sphere.  However, if $a/M \ll 1$ then they are close to the sphere
$\{r = 3M\}$. 

We now recall the setup and results from \cite{TT} for the Kerr
spacetime. Let $\tau, \xi, \Phi$ and $\Theta$ be the Fourier variables
corresponding to $t, r, \phi$ and $\theta$, and
\[
\begin{split}
  p_{\mathbf K}( r,\theta,\tau, \xi, \Phi,\Theta) = & \
  g^{tt}\tau^2+2g^{t\phi}\tau\Phi+g^{\phi\phi}\Phi^2 +g^{rr}\xi^2
  +g^{\theta\theta}\Theta^2 \\  = & \  g^{tt}(\tau -\tau_1( r, \theta,
  \xi, \Phi,\Theta)) (\tau -\tau_2( r, \theta, \xi, \Phi,\Theta))
\end{split}
\]
be the principal symbol of $\Box_{\mathbf K}$. Here $\tau_1$, $\tau_2$
are real distinct smooth $1$-homogeneous symbols. It is known that
all trapped null geodesics in $r>r_+$ satisfy
\begin{equation}\label{Reqn}
  R_a(r,\tau,\Phi)= 0 
\end{equation}
where
\[
R_a(r,\tau,\Phi) = (r^2+a^2)(r^3-3Mr^2+a^2r+a^2M)\tau^2 -
2aM(r^2-a^2)\tau\Phi - a^2(r-M)\Phi^2.
\]
 
Let $r_a(\tau,\Phi)$ be root of \eqref{Reqn} near $r = 3M$, which can
be shown to exist and be unique for small $a$. Then define the symbols
\[
c_i ( r, \theta, \xi, \Phi,\Theta) = r- r_a(\tau_i,\Phi), \qquad i =
1,2
\]
and the associated space-time norms:
 
 \[
 \| u\|_{L^2_{c_i}}^2 = \| c_i (D,x) u\|_{L^2}^2 + \|u\|_{H^{-1}}^2
 \]
 \[
 \| g\|_{c_i L^2}^2 = \inf_{c_i(x,D) g_1 + g_2 = g} (\|g_1\|_{L^2}^2 +
 \|g_2\|_{H^1}^2).
 \]
The replacements of the $LE$ and $LE^*$ norms are
 \[
 \begin{split}
   \|u\|_{LE_{\mathbf K}^1} = & \ \| \chi_{ps}(D_t - \tau_2(D,x))c_1 (D,x)\chi_{ps}
   u\|_{L^2}^2\\ & \  + \| \chi_{ps}(D_t - \tau_1(D,x))c_2 (D,x)\chi_{ps} u\|_{L^2}^2 
   \\ & \ +  \|\partial_r u\|_{LE}  +
   \|(1-\chi^2) \nabla u\|_{LE}+ \|r^{-1} u\|_{LE}
\\
 \|f\|_{LE_{\mathbf K}^*} = & \| (1-\chi_{ps}) f\|_{LE^*} + \|\chi_{ps} f\|_{c_1
   L^2 \cap  c_2 L^2}.
\end{split}
 \]
 Then the main result in \cite{TT} asserts that
 
\begin{theorem}\label{Kerr}
  Let $u$ solve $\Box_{\mathbf K} u = f$ in $\M$. Then
  \begin{equation}
 \|u \|_{LE^1_{{\mathbf K}}} +\sup_{\tilde v \geq 0} \| \nabla u(\tv)\|_{L^2}
 \lesssim\|\nabla u(0)\|_{L^2}  + \| f \|_{LE^*_{\mathbf K}}.
    \label{Kest}\end{equation}
 \end{theorem}

 This is the direct counterpart of \eqref{Sest}, which corresponds to
$\tau_1 = -\tau_2$ and
 $c_1 = c_2 = r-3M$. Again, the loss of one derivative can be 
improved to an $\epsilon$ loss, or even to a logarithmic loss, 
see \cite{Mihai}.

\subsection{Stationary local energy decay for small 
perturbations of Schwarzs\-child} \label{ss:sle} Here we consider a  Lorentzian metric $g$ in $\M$ 
which is a small perturbation of Schwarzschild expressed in
 the $(\tv,r,\phi,\theta)$ coordinates. Our main result is as follows:
 
\begin{theorem}\label{LEpert}
  Let $g$ be a Lorentzian metric on $\M$ and $u$ a smooth function in $\M$.

a)  Let $\chi_{ps}$  be a smooth cutoff function which selects a small neighbourhood of the photon sphere $\{ r = 3M\}$. If
  \begin{equation}\label{aproxmetinfty}
    |\partial_\alpha [g_{\mu\nu}-(g_{\mathbf K})_{\mu\nu}]| \lesssim \epsilon \la r\ra^{-|\alpha|-}, \qquad 0\leq |\alpha| \leq 1
\end{equation}
  for a small enough $\epsilon$, then the stationary local energy
  bound holds for $u$:
  \begin{equation}\label{st.est}
    \| u\|_{LE^1[\tv_0,\tv_1]} \lesssim \|\nabla u(\tv_0)\|_{L^2} + \|\nabla u(\tv_1)\|_{L^2}+ 
    \| \Box_g u \|_{LE^*[\tv_0,\tv_1]} + \|\chi_{ps} \partial_{\tv} u\|_{LE[\tv_0,\tv_1]}.
  \end{equation}

b) If in addition  
\begin{equation}\label{aproxmetinfty2}
 |\partial^\alpha g^{\mu\nu}| \lesssim r^{-1-}, \qquad 1 \leq |\alpha| \leq k+1 
\end{equation}
then \eqref{sle} also holds.

\end{theorem}

\begin{proof}
  a) The argument is based on the computation in \cite{MMTT}, which
  only requires the use of vector fields. We will first prove the
  estimate in the Schwarzschild case, but do it in such a way so that
  the transition to $g$ is perturbative.

  Let $X$ be a differential operator
  \begin{equation}\label{Xdef}
    X = b(r)\partial_r + c(r)\partial_{\tv} + q(r)
  \end{equation}
  for some smooth functions $b, c, q : [r_e, \infty) \to \R$ 
with $c$ constant outside a compact region and $b$, $q$ satisfying
\begin{equation}\label {condonX}
    \begin{split}
    |\partial_r^\alpha b| \leq c_\alpha r^{-|\alpha|} \\
      |\partial_r^\alpha q| \leq c_\alpha r^{-1-|\alpha|}.
    \end{split}
 \end{equation} 
  Let $\M_{[\tv_0,\tv_1]} = \{ \tv_0 < \tv < \tv_1,\ r > r_e \}$, and let
  $dV_{\mathbf S}=r^2 dr d\tv d\omega$ denote the Schwarzschild
  induced measure. It was shown in \cite{MMTT} that one can find $X$
  as above so that
  \begin{equation}\label{intdivs}
    \int_{\tv_0}^{\tv_1} Q^{\mathbf S}(\tv) d\tv = 
    - \int_{\M_{[\tv_0,\tv_1]}} \Box_{\mathbf S} u \cdot  Xu \  dV_{\mathbf S}
    - \left. BDR^{\mathbf S}[u]\right|_{\tv = \tv_0}^{\tv = \tv_1} - \left. BDR^{\mathbf S}[u] \right|_{r=r_e}
  \end{equation} 
  with
  \[
  \int_{\tv_0}^{\tv_1} Q^{\mathbf S}(\tv) d\tv \gtrsim 
\| u\|_{LE^1_{{\mathbf S},w}[\tv_0,\tv_1]}^2
  \]
where
\[
\| u\|_{LE^1_{{\mathbf S},w}} = \| (r-3M) \chi \nabla 
   u\|_{L^2}^2 +
 \|r^{-2}\partial_r u\|_{L^2}^2 + \|r^{-1} u\|_{L^2}^2 +
   \|(1-\chi^2) \nabla u\|_{L^2}^2
\]
 and the boundary terms satisfy
  \[
\begin{split}
  \left. BDR^{\mathbf S}[u]\right|_{\tv = \tv_i} \approx & \  \|\nabla u(\tv_i)\|_{L^2}^2,
  \qquad i=1,2
  \\
  \left. BDR^{\mathbf S}[u]\right|_{r=r_e} \approx & \ \| u\|_{H^1(\Sigma^+_{[\tv_0,\tv_1]})}^2.
\end{split}  
\]
Comparing the $LE_{{\mathbf S},w}^1$ norm we obtain from this
computation with the $LE^1$ norm which we need, one sees that two
improvements are necessary, one near $r= 3M$ and another for large
$r$.

The improvement for large $r$ is a consequence of the fact that for large $r$ 
one can view the Schwarzschild metric as a small perturbation of 
the Minkowski metric. Precisely, from the estimate \cite{MT1}(2.3)
(see also \cite{MT}) we have for large $R$
\begin{equation}
\| \chi_{> R} u\|_{LE_1[\tv_0,\tv_1]}\lesssim  \| u\|_{LE^1_{{\mathbf S},w}[\tv_0,\tv_1]} +
\|\chi_{>R} \Box_{\mathbf S} u \|_{LE[\tv_0,\tv_1]}+ \|\nabla u(\tv_0)\|_{L^2}.
\label{highr}\end{equation}
The similar bound for the metric $g$ is also valid.

To gain the improvement for $r$ close to $3M$ we add a Lagrangian
correction term. Precisely, a direct integration by parts yields
\[
\int_{\M_{[\tv_0,\tv_1]}} \!\!\!\!\!\!\!\! \Box_{\mathbf S} u \cdot \chi_{ps} ^2 u \ dV_{\mathbf S} =
\int_{\M_{[\tv_0,\tv_1]}}\!\!\!\!\!\!  \chi_{ps}^2 g_{\mathbf S}^{\mu\nu} \partial_{\mu} u \partial_{\nu} u +
(\Box_{\mathbf S} \chi_{ps}^2) u^2\ dV_{\mathbf S} + \left.\int_{\M_\tv} \!\!\!\!  \chi_{ps}^2 g_{\mathbf S}^{00}
  u \partial_t u dx \right|^{\tv = \tv_1}_{\tv = \tv_0}.
\]
Since $\partial_t$ is timelike in the support of $\chi_{ps}$, we can write 
the pointwise bound
\[
|\chi_{ps}  \nabla u|^2 \lesssim \chi_{ps}^2 g_{\mathbf S}^{\mu\nu} \partial_{mu} u \partial_{\nu} u + C|\chi_{ps} \partial_t u|^2
\]
for some large constant $C$.  Then the previous identity 
yields
\begin{equation} \label{r=3M}
\| \chi_{ps}  \nabla u\|_{L^2(\M_{[\tv_0,\tv_1]})}^2 \lesssim 
\int_{\M_{[\tv_0,\tv_1]}} \!\!\!\!\!\!\!\! 
\Box_{\mathbf S} u \cdot \chi_{ps}^2 u \ dV_{\mathbf S} + C \| \chi_{ps} \partial_t u\|_{L^2(\M_{[\tv_0,\tv_1]})}^2 + 
\sum_{i=1,2} \|\nabla u(\tv_i)\|_{L^2}.
\end{equation}

Combining \eqref{intdivs}, \eqref{highr} and \eqref{r=3M} we obtain
\begin{equation}\label{good-S}
\begin{split}
\  \|  u\|_{LE^1(\M_{[\tv_0,\tv_1]})}^2
& + \| u\|_{H^1(\Sigma^+_{[\tv_0,\tv_1]})}^2 + \|\nabla u(\tv_1)\|_{L^2} ^2 \lesssim
- \int_{\M_{[\tv_0,\tv_1]}} \!\!\!\!\!\!\!
\Box_{\mathbf S} u \cdot  X_1 u \  dV_{\mathbf S} 
  \\ &   + \|\nabla u(\tv_0)\|_{L^2} ^2   
+C \| \chi_{ps} \partial_t  u\|_{L^2(\M_{[\tv_0,\tv_1]})}^2
 + \|\chi_{>R} \Box_{\mathbf S} u\|_{LE^*(\M_{[\tv_0,\tv_1]})}^2
\end{split}
\end{equation}
where 
  \[
  X_1 = X+\delta \chi^2_{ps}(r) 
  \]
  with a fixed small constant $\delta$.  At this point, the desired
  conclusion \eqref{st.est} for the Schwarzschild metric follows
  if we  estimate the integral term by $\| \Box_{\mathbf S} u\|_{LE^*}
  \|u\|_{LE^1}$.

  It remains to show that a similar estimate holds with $\Box_{\mathbf S}$
  replaced by $\Box_g$.  This substitution is easily made in the last
  term on the right by performing a similar substitution in
  \eqref{highr}. Consider now the difference in the integral term,
\[
D= \int_{\M_{[\tv_0,\tv_1]}} 
(\Box_{\mathbf S}-\Box_g) u \cdot  X_1 u \  dV_{\mathbf S}.
\]
To estimate this we use the bound \eqref{aproxmetinfty} to write 
\[
\Box_{\mathbf S}-\Box_g = \partial_{\mu} (g^{\mu\nu}_{\mathbf S} - g^{ij}) \partial_{\nu} + O(\epsilon r^{-1-}) \nabla.
\]
Then we integrate by parts in a standard manner. Using also Hardy
type inequalities we obtain
\[
|D| \lesssim \epsilon (\sum_{i=1,2}
\|\nabla u(\tv_i)\|_{L^2}^2 + \| u\|_{H^1(\Sigma^+_{[\tv_0,\tv_1]})}^2
+ \|u\|_{LE^1(\M_{[\tv_0,\tv_1]})}^2).
\]
Hence \eqref{good-S} for $\Box_g$ follows, and the proof 
of the stationary local energy bound \eqref{st.est} is concluded.

b) The proof follows closely that of Theorem 4.5 in \cite{TT}, but for
the sake of completeness we include it here.  The result will follow
by induction on $k$. The case $k=0$ is part (a) of the theorem. We
will prove the case $k=1$, and the rest follows in a similar manner.

We need to estimate $\| \nabla^2 u\|_{LE[\tv_0,\tv_1]}$. We already control
 $\| \nabla \partial_\tv u\|_{LE[\tv_0,\tv_1]}$, therefore it remains to estimate 
the second order spatial derivatives. We write 
\[
\Box_g u = L_1 \partial_\tv u  + L_2 u
\]
where $L_1$ is a first order operator and $L_2$ is a purely spatial second 
order operator. Then we have 
\[
\|  L_2 u \|_{LE[\tv_0,\tv_1]} \lesssim \|  \Box_g u \|_{LE[\tv_0,\tv_1]} + 
\| L_1 \partial_\tv u\|_{LE[\tv_0,\tv_1]} 
\]
which is favorable where $L_2$ is elliptic. But $L_2$ is elliptic wherever 
$\partial_\tv$ is time-like. Since $g$ is a small perturbation
of $g_{\mathbf S}$,  this happens everywhere outside a small neighbourhood
of $r = 2M$. Thus we have the elliptic estimate
\begin{equation}\label{h2out}
  \| \chi_{out} \nabla^2 u \|_{LE[\tv_0,\tv_1]} \lesssim  \|  \Box_g u \|_{LE[\tv_0,\tv_1]} + 
\| \nabla \partial_\tv u\|_{LE[\tv_0,\tv_1]}  +  \| \nabla u\|_{LE[\tv_0,\tv_1]}
\end{equation}
where $\chi_{out}$ selects the region $\{ r > 2M+\delta\}$.

It remains to estimate $\| \nabla^2 \chi_{eh} u\|_{L^2}$ where 
$\chi_{eh}$ is a smooth cutoff function which selects the region 
$ \{ r < 2M+ 2\delta\}$ near the event horizon. 
The function $v = \chi_{eh} u$ solves the equation
\[
\Box_g v = h:= \chi_{eh} f + [\Box_g,\chi_{eh}]
\]
where the second term on the right is controlled in $H^1$ 
via \eqref{h2out}. Hence the conclusion of part (b) of the Proposition 
would follow from the following 

\begin{lemma}\label{l:eh}
  Let $\M_{eh}= \{ r_e  < r < 2M+3\delta, \ \tv \geq 0\}$ with a
  fixed small $\delta$.  Let $g$ be an $O(\epsilon)$ perturbation of
  $g_{\mathbf S}$ in $C^{m+1}(\M_{eh})$ with $\epsilon$ sufficiently small. Then
  for all functions $v$ with support in $\{ r < 2M+3\delta\}$ we have
\begin{equation}
\|  \nabla v\|_{H^{m}(\M_{eh})} \lesssim \| \nabla v(0)\|_{H^m} + 
\|  \Box_g v\|_{H^m(\M_{eh})} 
\label{ev-energy}
\end{equation} 
The similar estimate holds in any interval $[\tv_0,\tv_1]$.
\end{lemma}

\begin{proof}
  This is an estimate which is localized near the event horizon, and
  we will prove it taking advantage of the red shift effect. In microlocal terms,
the red shift effect is equivalent to exponential energy decay along 
the light rays which stay on the event horizon, and small perturbations thereof.
But for this estimate, these are all light rays of interest. All others 
exit the domain $\M_{eh}$ in a finite time.

We begin with a simplification. If $\epsilon$ is small enough then for 
$r < 2M - \delta$ the $r$ spheres are uniformly time-like, therefore 
we can use standard local  energy estimates for the wave equation 
to reduce the problem to the case when $r_e =  2M - 2\delta$.

For $m = 0$ the above bound follows from part (a) of the Proposition.
For $m=1$ we commute $\Box_g$ with the vector fields $\partial_\tv$, 
$\Omega$ and $\partial_r$.  We have
\[
[\partial_\tv,\Box_g] = O(\epsilon) Q_2 , \qquad [\Omega,\Box_g] = O(\epsilon) Q_2
\]
for some second order partial differential operator $Q_2$ with bounded
coefficients.  Hence applying \eqref{ev-energy} with $m=0$ to
$\partial_\tv v$ and $\Omega v$ we obtain
\begin{equation}\label{tovf}
\| \Omega v\|_{H^1(\M_{eh})} + \| \partial_\tv v\|_{H^1(\M_{eh})}
 \lesssim \|h\|_{H^1(\M_{eh})}  
+ \epsilon \| v\|_{H^2(\M_{eh})}.
\end{equation}

  We still need to bound $\partial_{r} v$. For that we compute the
  commutator
  \begin{equation}
    [\Box_g, \partial_r]= -(\partial_r g_{\mathbf S}^{rr})\partial_{r}^2  +  O(\epsilon) Q_2 + N_2
    \label{drcom}\end{equation}
  where $N_2$ stands for a second order operator with no $\partial_r^2$
  terms. The key observation here is that $\gamma=\partial_r g^{rr}_{\mathbf S} >
  0$ near $r = 2M$. We can now write
  \[
  (\Box_g - \gamma_1 X )\partial_r v = \partial_r h + (O(\epsilon) Q_2 + N_2) v, 
\qquad \gamma_1 > 0
 \]
  with $N_2$ as above and most importantly, a positive coefficient
  $\gamma_1$.  We recall here that $X$ looks like $-\partial_r$ near
  the event horizon.   Because of this the operator
  \[
  B= \Box_g - \gamma_1 X
  \]
  satisfies the same estimate  \eqref{st.est} as $\Box_g$ for
  functions supported near the event horizon. To see this it suffices 
to examine  \eqref{good-S}  with $\Box_{\mathbf S}$ replaced by $\Box_g$. 
Since $X = X_1$ near the event horizon, it follows that the contribution of 
$\gamma_1 X$ has the right sign and can be discarded.
Hence we obtain
\begin{equation}\label{rvf}
\begin{split}
\|\partial_r v\|_{H^1(\M_{eh})} 
 \lesssim & \ \|h\|_{H^1(\M_{eh})}  
+ \epsilon \| v\|_{H^2(\M_{eh})} \\ & \ 
+ \| \Omega v\|_{H^1(\M_{eh})} + \| \partial_\tv v\|_{H^1(\M_{eh})}+ \|  v\|_{H^1(\M_{eh})}
\end{split}
\end{equation}
where the last three terms account for the effect of $N_2$.
Combining the bounds \eqref{tovf} and \eqref{rvf} we obtain
\[
\| v\|_{H^2(\M_{eh})} 
 \lesssim \|h\|_{H^1(\M_{eh})}  
+ \epsilon \| v\|_{H^2(\M_{eh})}+
 \|  v\|_{H^1(\M_{eh})}.
\]
If $\epsilon$ is sufficiently small then the conclusion \eqref{ev-energy}
follows for $m = 1$. The argument for $m > 1$ is similar.
\end{proof}
\end{proof}

\subsection{Local energy decay for 
small perturbations of  Kerr} \label{ss:le}

Here we consider small perturbations of a Kerr space-time with small
angular momentum, $|a| \ll M$. Our main result is as follows:

\begin{theorem}\label{LEpertw}
  Let $g$ be a Lorentzian metric on $\M$ and $u$ a function in $\M$
  solving $\Box_g u = f$.

  a) Assume that $g$ satisfies \eqref{aproxmetinfty} and decays to
  $g_{\mathbf K}$ near the photon sphere,
  \begin{equation}\label{aproxmetcpt}
    \chi_{ps} |\partial_\alpha [g_{\mu\nu}-(g_{\mathbf K})_{\mu\nu}]|\leq c_\alpha (\tv), 
\qquad 0\leq |\alpha| \leq 1 
  \end{equation} 
  where $c_\alpha \in L_{\tv}^1$ (in particular, we can take $c_\alpha
  = \la \tv \ra ^{-1-}$). Then the weak local energy estimate holds:
  \begin{equation}\label{lee}
    \|u\|_{H^1(\Sigma_R^+)}^2 +\sup_{\tilde t \geq 0}
    \|\nabla u(\tilde t)\|_{L^2}^2 +  \|u\|_{LE^1_{\mathbf K}}^2 \lesssim \|\nabla u(0)\|_{L^2}^2 
 +    \|f\|_{LE^*_{\mathbf K}}^2.
  \end{equation}

b) Assume in addition that \eqref{aproxmetinfty2} holds. Then
\begin{equation}\label{leek}
   \|u\|_{H^k(\Sigma_R^+)}^2 +\sup_{\tilde t \geq 0}
    \|\nabla u(\tilde t)\|_{H^k}^2 +  \|u\|_{LE^{1,k}_{\mathbf K}}^2 \lesssim \|\nabla u(0)\|_{H^k}^2 
 +    \|f\|_{LE^{*,k}_{\mathbf K}}^2,
  \end{equation}
and thus \eqref{lew} holds.
\end{theorem}
\begin{proof}
 

On any fixed compact time interval we have uniform energy estimates. 
Eliminating a compact time interval, we can assume without any 
restriction in generality that the integrability condition 
on $c(\tv)$  is strengthened to
 \begin{equation}\label{ceps}
 \int_0^\infty c(\tv) d\tv \lesssim \epsilon, \qquad |c(\tv)| \lesssim \epsilon
  \end{equation} 
where \eqref{aproxmetinfty} was also used.

  a) The proof of \eqref{lee} is similar to the proof of
  Theorem~\ref{LEpert}, but it requires the use of
  pseudodifferential operators. Let us start by recalling the idea
  behind the proof of Theorem~\ref{Kerr}. It is shown in \cite{TT}
  that there exists a pseudodifferential operator $S_1$ of order $1$
  that satisfies the following:
 
  a) $S_1$ is a differential operator in $\tv$ of order $1$.
 
  b) $S_1 = X + \chi_{ps} s^w \chi_{ps}$, where $X$ is defined as in
  \eqref{Xdef} and $s \in S^1$.
  
  c) Let $\M_{[0,\tv_0]} = \{ 0 < \tv < \tv_0,\ r > r_e \}$ and
  $dV_{\mathbf K}=\rho^2 dr d\tv d\omega$ denote the Kerr induced
  measure. Then one has
 \begin{equation}
   \int_{0}^{\tv_0} Q^{\mathbf K}(\tv) d\tv = 
   - \int_{\M_{[0,\tv_0]}} (\Box_{\mathbf K} u) (S_1 u) dV_{\mathbf K}
   - \left. BDR^{\mathbf K}[u]\right|_{\tv = 0}^{\tv = \tv_0} - \left. BDR^{\mathbf K}[u] \right|_{r=r_e}
   \label{intdivk}\end{equation}
 with
 \[
 \int_{0}^{\tv_0} Q^{\mathbf K}(\tv) d\tv \gtrsim \| u\|_{LE^1_{{\mathbf K},w}[0,
   \tv_0]}^2
 \]
\[
 \begin{split}
   \|u\|_{LE_{\mathbf K,w}^1}^2 = & \ \| \chi_{ps}(D_\tv - \tau_2(D,x))c_1 (D,x)\chi_{ps}
   u\|_{L^2}^2 \\ & \ + \| \chi_{ps}(D_\tv - \tau_1(D,x))c_2 (D,x)\chi_{ps} u\|_{L^2}^2 \\
&  \ +
 \|r^{-2}\partial_r u\|_{L^2}^2 + \|r^{-1} u\|_{L^2}^2 +
   \|(1-\chi_{ps}^2) \nabla u\|_{L^2}^2
\end{split}
\]
and the boundary terms satisfying
 \begin{equation}\label{bdrposk}
   \begin{split}
     \left. BDR^{\mathbf K}[u]\right|_{\tv = \tv_i} \approx \| \nabla u(\tv_i)\|_{L^2} \\
     \left. BDR^{\mathbf K}[u]\right|_{r=r_e} \approx \| u\|_{H^1(\Sigma^+_{[0,t_0]})}^2
   \end{split}
 \end{equation}
 Note that conditions a) and b) guarantee that the boundary terms are
 well-defined after integrating by parts.

The same reasoning as in Theorem~\ref{LEpert} leads to the counterpart 
of \eqref{good-S}, namely 
 \begin{equation}\label{good-K}
\begin{split}
& \|  u\|_{LE^1_{\mathbf K}(\M_{[0,\tv_0]})}^2 + \|\nabla u\|_{L^2(\Sigma^+_{[0,\tv_0]})}^2 
+ \|\nabla u(\tv_1)\|_{L^2} ^2
\lesssim \\  
&- \int_{\M_{[0,\tv_0]}} \!\!\!\!\!\!\!
\Box_{\mathbf K} u \cdot  S_1 u \  dV_{\mathbf K}+  \|\nabla u(\tv_0)\|_{L^2} ^2 
 + \|\chi_{>R} \Box_{\mathbf K} u\|_{LE^*(\M_{[0,\tv_0]})}^2.
\end{split}
\end{equation} 
Here we seek to replace the Kerr metric by $g$. 
As discussed in Theorem~\ref{LEpert}, the bound \eqref{highr} holds 
as well for the metric $g$, therefore
the last term is not an issue. Hence it remains to consider the difference
\[
D = \int_{\M_{[0,\tv_0]}} \!\!\!\!\!\!\!
(\Box_{\mathbf K}-\Box_{g}) u \cdot  S_1 u \  dV_{\mathbf K}
\]
 We split $S_1 = X_1+ S_{ps}$ where 
\[
X_1 = X - \chi_{ps} X \chi_{ps}, \qquad S_{ps} = \chi_{ps} X \chi_{ps}+ 
\chi_{ps} s^w \chi_{ps}.
\]
Thus $X_1$ is a first order differential operator which is supported away from 
the photon sphere. Correspondingly we split $D = D_1+ D_{ps}$.
For $D_1$, integration by parts using \eqref{aproxmetinfty} leads to
\[
|D_1| \lesssim \epsilon (\|\nabla u(0)\|_{L^2}^2+\|\nabla u(\tv_0)\|_{L^2}^2 
+ \| u\|_{H^1(\Sigma^+_{[0,\tv_0]})}^2
+ \|u\|_{LE^1_{{\mathbf K}}(\M_{[0,\tv_0]})}^2).
\]
Here it is essential that the outcome of the integration by parts is supported 
away from the photon sphere $\{r = 3M\}$, where the 
$LE^1_{{\mathbf K}}$ and $LE^1$
norms are equivalent.

To estimate $D_{ps}$ we need to use the stronger bound \eqref{aproxmetcpt}.
Then near the photon sphere we can write 
\[
\Box_{\mathbf K} - \Box_g = \frac{1}{\sqrt{ |g_{\mathbf K}|}}\partial_{\mu} \sqrt{|g_{\mathbf K}|}
(g_{\mathbf K}^{\mu\nu}- g^{\mu\nu}) \partial_{\nu} + O(c(\tv)) \nabla 
\]
where we also have $g_{\mathbf K}^{\mu\nu}- g^{\mu\nu}=  O(c(\tv))$ and $\nabla (g_{\mathbf K}^{\mu\nu}- g^{\mu\nu})=  O(c(\tv))$.
Thus integrating by parts we obtain
\[
|D_{ps}| \lesssim \int_{\M_{[0,\tv_0]}} \tilde \chi_{ps} c(\tv)  (|\nabla u|^2 + |u|^2) dV_{\mathbf K} 
+ c(0) \|\nabla u(0)\|_{L^2} ^2 + c(\tv_0) \|\nabla u(\tv_0)\|_{L^2} ^2.
\]

By \eqref{ceps} it follows that 
\[
|D_{ps}| \lesssim \epsilon \Bigl( \sup_{\tv \in [0,\tv_0]} \| \nabla u(\tv)\|_{L^2}^2 + \|  u\|_{LE^1_{\mathbf K}(\M_{[0,\tv_0]})}^2 \Bigr).
\]

Applying the bounds for $D_1$ and $D_{ps}$, we complete the replacement of 
$\Box_{\mathbf K}$ by $\Box_g$ in \eqref{good-K}, obtaining 
\begin{equation}
\label{good-K1}
\begin{split}
& \|  u\|_{LE^1_{\mathbf K}(\M_{[0,\tv_0]})}^2 + \|\nabla u\|_{L^2(\Sigma^+_{[0,\tv_0]})}^2 
+ \|\nabla u(\tv_0)\|_{L^2} ^2
\lesssim - \int_{\M_{[0,\tv_0]}} \!\!\!\!\!\!\!
\Box_{g} u \cdot  S_1 u \  dV_{\mathbf K} \\  
&+  \|\nabla u(0)\|_{L^2} ^2 
 + \|\chi_{>R} \Box_{g} u\|_{LE^*(\M_{[0,\tv_0]})}^2 
+ \epsilon \sup_{\tv \in [0,\tv_0]} \| \nabla u(\tv)\|_{L^2}^2
\end{split}
\end{equation}
To conclude the proof of \eqref{lee} we estimate the integral term in
the last inequality by $ \| u\|_{LE^1_{\mathbf
    K}(\M_{[0,\tv_0]})}\|\Box_{g} u\|_{LE^*_{\mathbf K}}$ and use
Cauchy-Schwarz.  The last term on the right is eliminated by taking
the suppremum in the resulting estimate over $\tv_0 \in [0,\tv_1]$
for arbitrary $\tv_1 > 0$.

b) We prove the estimate \eqref{leek} for $k=1$; the argument for larger
$k$ is identical. We begin by applying the estimate \eqref{lee} 
to $\partial_\tv u$.  Commuting $\Box_g$ with $\partial_\tv$ we have
\[
\Box_g \partial_\tv u = \partial_\tv \Box_g u +
O(\epsilon r^{-1-})[(1-\chi_{ps}) Q_2 u + \chi_{ps} Q_1 u]
+ O(c(\tv)) \chi_{ps} Q_2 u
\]
where $Q_1$ and $Q_2$ stand for second order operators with bounded
coefficients.  Here we have used \eqref{aproxmetinfty} for first
derivatives of $g$ away from the photon sphere, \eqref{aproxmetcpt}
for first derivatives of $g$ near the photon sphere, and
\eqref{aproxmetinfty2} for second order derivatives of $g$.
We estimate the second term in $\Box_g \partial_\tv u$ in $LE_{\mathbf K}^*$ 
and the third in $L^1_\tv L^2_x$. This gives
\[
\| \partial_\tv u\|_{LE_{\mathbf K}^1} \lesssim \epsilon \|u\|_{LE_{\mathbf K}^{1,1}} + \|u\|_{LE_{\mathbf K}^{1}}+
\| \Box_g u\|_{LE_{\mathbf K}^{*,1}}.
\]
Away from the event horizon the vector field $\partial_\tv$ is timelike,
therefore, arguing as in the proof of Theorem~\ref{LEpert}(b), we can
use an elliptic estimate to conclude that 
\[
\| \chi_{out} u\|_{LE_{\mathbf K}^{1,1}} \lesssim \epsilon \|u\|_{LE_{\mathbf K}^{1,1}}+ \|u\|_{LE_{\mathbf K}^{1}} +
\| \Box_g u\|_{LE_{\mathbf K}^{*,1}}.
\]
On the other hand, near the event horizon we use Lemma~\ref{l:eh}
to obtain
\[
\| \chi_{eh} u\|_{H^2} \lesssim  \|\chi_{out} u\|_{H^2} +
\| \chi_{eh} \Box_g u\|_{H^1} + \| \nabla u(0)\|_{H^1}
\]
Combining the last three estimates we obtain \eqref{leek} for $k=1$.
\end{proof}

\subsection{Conclusion}
We can now prove Price's law for certain perturbations of the Kerr
spacetimes:

\begin{theorem}\label{Pricepert}
  Let $g$ be a Lorentzian metric close to $g_{\mathbf K}$ in the sense
  that it satisfies \eqref{aproxmetinfty}, \eqref{aproxmetinfty2} and
  \eqref{aproxmetcpt}. Let $u$ solve \eqref{box} with smooth,
  compactly supported initial data and $V=0$. Then
  \eqref{pointwiseest1} holds.
\end{theorem}
\begin{proof}
  This is an obvious consequence of Theorems \ref{main},
  \ref{LEpert} and \ref{LEpertw}.
\end{proof}


\bibliography{Kerrpert}
\bibliographystyle{plain}

\end{document}